\documentclass[a4paper]{amsart}
\usepackage{amssymb}
\usepackage[arrow, matrix, curve]{xy}
\usepackage{comment}
\usepackage{stmaryrd}
\usepackage{color}
\usepackage{hyperref}

\DeclareMathOperator{\card}{card}

\newcommand\E   	{\mathcal{E}}
\newcommand\F   	{\mathcal{F}}
\newcommand\Hom		{\mathrm{Hom}}
\newcommand\Sub		{\mathrm{Sub}}
\newcommand\dSet	{\text{dSet}}
\newcommand\Set     {\text{Set}}
\newcommand\sSet    {\text{sSet}}
\newcommand\Oper    {\text{Oper}}

\newcommand{\ds}{\displaystyle}
\def\to{\rightarrow}

\newtheorem{theorem}{Theorem}[section]
\newtheorem{lemma}[theorem]{Lemma}
\newtheorem{prop}[theorem]{Proposition}
\newtheorem{cor}[theorem]{Corollary}

\theoremstyle{definition}
\newtheorem{definition}[theorem]{Definition}
\newtheorem{example}[theorem]{Example}

\theoremstyle{remark}
\newtheorem{remark}[theorem]{Remark}

\numberwithin{equation}{section}

\begin{document}

\title[On the combinatorics of anodyne extensions of dendroidal sets]{On the combinatorics of faces of trees and \\ anodyne extensions of dendroidal sets}


\author[M. Ba\v{s}i\'{c}]{Matija Ba\v{s}i\'{c}}
\address{Department of Mathematics, Faculty of Science, University of Zagreb, Croatia}
\email{mbasic@math.hr}

\subjclass[2000]{Primary: 55U05. Secondary: 55P48, 18G30.}

\keywords{dendroidal sets, anodyne extensions, pushout-product property}

\date{}

\dedicatory{}

\begin{abstract}
We discuss the combinatorics of faces of trees in the context of dendroidal sets and develop a systematic treatment of dendroidal anodyne extensions. As the main example and our motivation, we prove the pushout-product property for the stable model structure on dendroidal sets.
\end{abstract}

\maketitle

\section*{Introduction}

Combinatorial aspects of simplicial homotopy theory are governed by the simplicial identities for face and degeneracy maps. The main objects of study, the Kan complexes, are simplicial sets having the right lifting property with respect to horn inclusions, where a horn $\Lambda^k[n]$ of a simplex $\Delta[n]$ is a union of all but one of its faces. Simplicial anodyne extensions were first introduced in \cite{GZ} as a saturated class of monomorphisms of simplicial sets generated by the horn inclusions. In particular it is shown that the same class is generated by a set of maps 
\[
\Lambda^k [n]   \times \Delta[m] \cup \Delta[n]\times \partial \Delta[m]  \to \Delta[n]\times \Delta[m]. 
\]
This property, sometimes called the pushout-product property, 
simplifies combinatorial arguments involving lifting properties and it is reflected in the existence of a Quillen model structure (or, in other words, of "a homotopy theory") on simplicial sets with anodyne extensions as acyclic cofibrations and Kan complexes as fibrant objects. 

In recent years, the theory of simplicial sets has been extended to a richer theory of dendroidal sets (\cite{den set}, \cite{inn Kan in dSet}). By considering linear orders as linear trees, one extends the simplex category $\Delta$ to a larger category $\Omega$ of all finite rooted trees. Dendroidal sets are presheaves on the category $\Omega$ and the theory is developed in a similar fashion as that of simplicial sets. In a series of papers (\cite{dSet model hom op}, \cite{den Seg sp}, \cite{dSet and simp ope})  D.-C. Cisinski and I. Moerdijk show that the category of dendroidal sets is endowed with a Quillen model structure such that the fibrant objects are exactly the $\infty$-operads (the operadic analogues of $\infty$-categories). They also show that this model structure is Quillen equivalent to the model structure on simplicial operads (generalizing the equivalence between the Joyal model structure on simplicial sets and Bergner model structure on simplicial categories). We will refer to this Quillen structure on dendroidal sets as the operadic model structure.

Further research (\cite{Heuts}, \cite{Heuts2}, \cite{BM}) shows that the operadic model structure admits a left Bousfield localization, called the covariant model structure, which is Quillen equivalent to $E_\infty$--spaces. Moreover, joint work with T. Nikolaus (\cite{stable}) shows that there is a further Bousfield localization, called the stable model structure, by which dendroidal sets model grouplike $E_\infty$--spaces (or, equivalently, connective spectra). These model structures generalize the mentioned model structure on simplicial sets. Similarly to simplices, dendrices (tree-like cells) have faces and horns, and hence there are natural notions of dendroidal Kan complexes and of dendroidal anodyne extensions. Nonetheless, the combinatorial arguments needed to establish model structures in the dendroidal settting are more intricate and the generalizations are rarely direct. 

In this paper we advance in the study of the combinatorial aspects of (operadic, covariant and stable) dendroidal anodyne extensions. In particular, for a tree $T$ and a subset $A$ of a representable dendroidal set $\Omega[T]$, we study the set of face maps $\partial_f \colon \partial_f P \to P$ (with $P$ a subtree of $T$) which do not factor through $A$. Such a set of face maps is called an extension set for $A$ if it satisfies five easy-to-check conditions (see Definition \ref{def extension set}). Our main result, Theorem \ref{theorem method}, roughly states: 

\noindent \textbf{Theorem.} \emph{If there is an extension set for $A$, then the inclusion  $A \to \Omega[T]$
	is a dendroidal anodyne extension. }\\

We consider this to give a combinatorial technique that simplifies various proofs in the theory. We show how this technique immediately applies to obtain some already known results (e.g. Example \ref{SegalCore}, 
Remark \ref{PushoutProductOper}) and we also apply it to obtain new results. Our main new result is a variant of the pushout-product property for the stable model structure (see Theorem \ref{PushoutProductStab}):  

\noindent \textbf{Theorem.} Let $S$ and $T$ be trees and let $v$ be the bottom vertex of $S$. If $S$ or $T$ is linear or both $S$ and $T$ are open trees, then the morphism
$$\Lambda^{v}[S] \otimes \Omega[T] \cup \Omega[S] \otimes \partial\Omega[T] \to \Omega[S]\otimes \Omega[T]$$ is a stable anodyne extension. \\

This result has particular importance (which is our main motivation) as it enables us to construct the stable model structure in a more direct way than it was done in \cite{stable}, without referring to the covariant model structure.  A considerable advantage of giving an alternative construction of the stable model structure is that it provides a direct characterization of fibrations between fibrant objects, which we did not know how to show without the result of Theorem \ref{PushoutProductStab}. Also, the case where both maps are in the category of open dendroidal sets shows that the corresponding model structure on open dendroidal sets is compatible with the colax monoidal structure.

This paper is based on one chapter of the author's PhD thesis (\cite{thesis}). The material has been significantly rewritten in order to simplify the presentation and make the combinatorial conditions more natural, but this has not changed the underlying content and the main results as stated in the chapter of the thesis. Furthermore, the establishment of the stable model structure on dendroidal sets follows from the main results of the paper by standard techniques as written in the mentioned thesis. 

\subsection*{Organization.} After recalling the definition of trees and basic results about dendroidal sets in Section 1, we discuss the poset of faces of a tree and dendroidal identities in Section 2. In Section 3 we explain our combinatorial method. We axiomatize sufficient conditions for an inclusion of dendroidal sets to be contained in the saturated class generated by horn inclusions. In Section 4 we first recall the Boardman--Vogt tensor product of trees and finally prove the pushout-product property for the stable model structure on dendroidal sets. 

\subsection*{Acknowledgements.}  Many thanks are due to Ieke Moerdijk for encouraging me to do this project, as well as for many discussions, comments on the proofs and uncountable advice on how to improve my writing. I would also like to thank Thomas Nikolaus for helpful discussions and comments on the draft. Two subtleties have been corrected in the final version in the proof of Lemma \ref{lemmaAxioms345} and in Remark \ref{PushoutProductOper} after being pointed out to me by Luis A. Pereira. I am grateful to thorough referees for many comments that have improved the text.

\section{The formalism of trees and dendroidal sets}

\subsection{Definition of a tree}

\begin{definition} A \emph{tree} is a triple $(T, \leq, L)$ consisting of a finite non-empty set $T$, a partial order $\leq$ on $T$ and a subset $L$ of maximal elements of $T$ such that
\begin{itemize}
\item there is a unique minimal element $r\in T$, called the \emph{root} of $T$;
\item for every $e\in T$, the order $\leq$ induces a total order on the set $\{t \in T \mid t \leq e\}$, called the \emph{branch from $e$ to the root}.
\end{itemize}
\end{definition}

We usually denote such a triple $(T, \leq, L)$ simply by $T$. Elements of $T$ are called \emph{edges}. The elements of the set $L$ are called \emph{leaves}. \emph{Inner edges} are edges other than the root and the leaves. We define the \emph{height} of an edge $e$ as the number of elements of the branch from $e$ to the root.

For an edge $e$ which is not a leaf, the set $v$ of all of its immediate successors is called a \emph{vertex}. We say that $e$ is the \emph{output} of $v$. Elements of a vertex are also called \emph{inputs} of $v$. We say that an edge $e$ is \emph{attached} to a vertex $v$ if $e$ is the output or an input of $v$. A \emph{sibling} of an edge $e$ is any other edge $f$ such that $e$ and $f$ are both inputs of the same vertex. The unique vertex whose output is the root is called the \emph{bottom vertex}. We say that a vertex is a \emph{top vertex} if all of its inputs are leaves. 	An \emph{outer vertex} is either a top or a bottom vertex.  A top vertex may be empty and then it is called a \emph{stump}. Note that the union of the set of leaves and the set of stumps is in bijection with the set of maximal elements of $T$. A tree with no stumps is called an \emph{open} tree. 

A tree with exactly one vertex is called a \emph{corolla} and denoted $C_n$ where $n$ is the number of leaves. A tree all of whose vertices have exactly one input is called a \emph{linear tree} and denoted $L_n$ where $n\geq 0$ is the number of vertices. A tree with no vertices is called the \emph{unit tree} and it is denoted by $L_0$. 

To draw a tree on a paper we must put a total order on the inputs of every vertex. This gives additional structure to the tree called a \emph{planar structure}.

\begin{example} \label{tree}
Here is a picture of a (planar) tree with a root $r$, the set of leaves $L=\{a,b,d,f\}$, inner edges $c,e$, a stump $u$, another top vertex $w=\{a,b\}$ and a bottom vertex $v=\{c,d,e,f\}$.
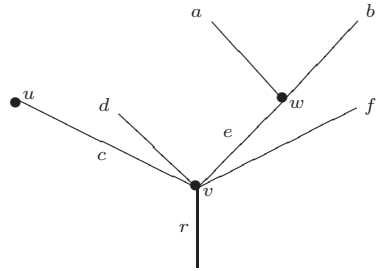
\[   \xymatrix{
&&& _{a}&&_{b} &&&& \\
& *=0{\,\, \bullet^u} & _d && *=0{\,\,\, \bullet_w} \ar@{-}[ul]  \ar@{-}[ur] &_f&& \\
&&&*=0{\,\, \bullet_v} \ar@{-}[ull]^c \ar@{-}[ul] \ar@{-}[ur]^e \ar@{-}[urr]    &&&&\\
&&&*=0{}\ar@{-}[u]^r &&&&
}
\]
\end{example}

\begin{definition} \label{grafting}
Let $S$ be a tree with set of leaves $L(S)=\{l_1,...,l_m\}$. Let $T_1,...,T_m$ be trees with pairwise disjoint underlying sets such that for every index $i\in\{1,...,m\}$ the root $l_i$ of $T_i$ is the only common element of $S$ and $T_i$.

We define a new tree $S\circ (T_1,...,T_m)$  such that
\begin{itemize}
\item the underlying set is the union $S\cup T_1 \cup ... \cup T_m$,
\item the partial order extends the partial orders of $S, T_1,...,T_m$ in the sense that $t \leq s$ for all $s\in S$ such that $l_i\leq s$ and all $t\in T_i, i=1,...,m$
\item the set of leaves is $L(T_1) \cup ...\cup L(T_m)$.
\end{itemize}
We say that we have obtained $S\circ(T_1,...,T_m)$ by \emph{grafting} the trees $T_1,...,T_m$ on top of $S$.
\end{definition}

\subsection{Operads associated with trees and the category $\Omega$}

\begin{definition}
Let $(T, \leq, L)$ be a tree, $n\geq 0$ an integer and $t_1,...,t_n, t$ elements of $T$ such that $t \leq t_i$ for $i=1,...,n$. A pair $(\{t_1,...,t_n\}; t)$ is an \emph{operation} of $T$ if 
\begin{itemize}
\item for every leaf $t \leq l$ there exists a unique $i\in \{1,2,...,n\}$ such that $t_i\leq l$;
\item for every stump $v$ with an output edge $t\leq e$ there exists at most one $i\in \{1,2,...,n\}$ such that $t_i\leq e$.
\end{itemize}
We also write $(t_1,...,t_n; t)$ for such an operation.
\end{definition}

\begin{remark}
Intuitively, an operation $(t_1,...,t_n; t)$ can be also thought of as a (connected) subtree of $T$ with leaves $t_1, \ldots, t_n$ and the root $t$. 
\end{remark}

\begin{example}
Let $v$ be a vertex of a tree $T$ with an output edge $e$. Then $(v,e)$ is an operation of $T$.
\end{example}

\begin{example} The tree  
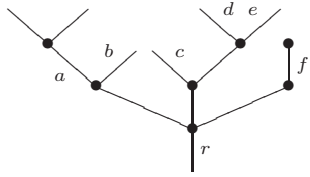
\begin{equation*}
\xymatrix@R=10pt@C=12pt{
&&&&&&&& \\
&*=0{  \bullet} \ar@{-}[ul]  \ar@{-}[ur]  &&&& *=0{ \bullet} \ar@{-}[ul]_d  \ar@{-}[ur]^e   & *=0{\bullet}  && \\
&& *=0{ \bullet} \ar@{-}[ul]^a  \ar@{-}[ur]^b && *=0{ \bullet} \ar@{-}[ul]_c  \ar@{-}[ur] && *=0{ \bullet} \ar@{-}[u]_f    \\
&&&&*=0{ \bullet}  \ar@{-}[ull] \ar@{-}[u] \ar@{-}[urr]    &&\\
&&&&*=0{}\ar@{-}[u]_{r} &&
}
\end{equation*}
has an operation $(a,b,c,d,e,f;r)$, an operation $(a,b,c,d,e; r)$ and many others.
\end{example}

\begin{lemma} \label{operations_lemma}
Let $T$ be a tree.
\begin{enumerate}
\item[(a)] For every $t\in T$, $(t;t)$ is an operation of $T$.
\item[(b)] If $(t_1,...,t_n;t)$ and $(t_{i,1},...,t_{i,k_i};t_i)$ for $i\in\{1,...,n\}$ are operations of $T$, then  $(t_{1,1},...,t_{1,k_1}, t_{2,1},..., t_{n,k_n};t)$ is also an operation of $T$.
\item[(c)] If $(t_1,...,t_n; t)$ is an operation of $T$ then $(t_{\sigma(1)}, t_{\sigma(2)}, ...., t_{\sigma(n)}; t)$ is also an operation of $T$ for any permutation $\sigma\in \Sigma_n$.
\end{enumerate}
\end{lemma}

\begin{proof} All statements follow directly from the definition and their verification is left to the reader.
\end{proof}

\begin{definition} \label{X-vertex}
	Let $(X; r_T)$ be an operation in $T$. We call a vertex $w$ of $T$ an $X$--\emph{vertex} if
	\begin{itemize}
		\item $w$ is non-empty and all inputs of $w$ are elements of $X$ or
		\item $w$ is empty and for its output $y$ there is no $x\in X$ such that $x\leq y$.
	\end{itemize}
\end{definition}

\begin{definition}
To every tree $T$ we associate a coloured operad $\Omega(T)$ with a set of colours being $T$ and
 \begin{equation*}
 \Omega(T)(t_1,...,t_n; t)=
 \left\{\begin{array}{ll}
 *, & \textrm{ if } (t_1,...,t_n; t) \textrm{ is an operation of } T;    \\
 \emptyset, & \textrm{ otherwise,}
 \end{array} \right.
 \end{equation*}
where $*$ denotes a fixed singleton.  The structure maps are uniquely determined and Lemma \ref{operations_lemma} shows they are well-defined.
\end{definition}

\begin{remark}
In the literature, $\Omega(T)$ is described equivalently as the free operad generated by the vertices of $T$.
\end{remark}

\begin{lemma} \label{operations_map}
Let $S$ and $T$ be trees. A map of sets $f:S\to T$ extends to a morphism of operads $f: \Omega(S) \to \Omega(T)$ if and only if $(f(s_1),...,f(s_n); f(s))$ is an operation in $T$ for every operation $(s_1,...,s_n; s)$ in $S$.
\end{lemma}

\begin{proof}
A morphism of operads $f: \Omega(S) \to \Omega(T)$ consists of component maps $\Omega(S)(s_1,...,s_n;s) \to \Omega(T)(f(s_1),f(s_2),...,f(s_n); f(s))$ compatible with the structure maps. The component maps are either the unique maps $\emptyset \to *$ or identities on $\emptyset$ or identities on $*$. Compatibility follows directly since all structure maps are uniquely determined by their domains and codomains.
\end{proof}

\begin{definition}
The category $\Omega$ of trees is a category whose objects are trees and the morphism sets are given by
\begin{equation*}
\Hom_{\Omega}(S,T) = \Hom_\Oper (\Omega(S), \Omega(T)).
\end{equation*}
So, $\Omega$ is by definition a full subcategory of the category $\Oper$ of (coloured) operads. 
\end{definition}

The category $\Omega$ is not skeletal, and in contrast to the category of linear orders and weakly monotone maps (whose skeleton is usually denoted $\Delta$) there is no natural choice for the representatives of isomorphism classes of objects in $\Omega$. Nonetheless, given a monomorphism $f\colon S\to T$ in $\Omega$, the image of $f$ is a tree $T'\subseteq T$ and there is a unique isomorphism $S\to T'$ such that $f$ factors as 
\[
\xymatrix{S \ar^{\cong}[r] & T' \ar@{^{(}->}[r] & T.}
\]

\begin{definition}
	A monomorphism $f\colon S\to T$ is a \emph{simple face map} if $S$ is equal (and not just isomorphic!) to the image $f(S)$. If $T$ has exactly one vertex more than $S$, we say that $f$ is an \emph{elementary face map}. 	
\end{definition}

\begin{remark} Elementary face maps are explicitly described and their relations are studied in Section \ref{elementary faces}. Simple face maps are exactly the compositions of elementary face maps. 
	Similarly, epimorphisms $S\to T$ such that $S$ has exactly one vertex more than $T$ are called \emph{elementary degeneracy maps}. 
\end{remark}

\begin{lemma}[\cite{den set}, 3.1]
	Every morphism in $\Omega$ can be factored in a unique way as an epimorphism  followed by a simple face map. Every epimorphism can be factored as a composition of elementary degeneracy maps followed by an isomorphism. 
\end{lemma}

\subsection{Dendroidal sets}

\begin{definition}
A \emph{dendroidal set} is a presheaf on the category $\Omega$. We denote the category of dendroidal sets by
$$\dSet := [\Omega^{\textrm{op}} , \Set].$$
\end{definition}

We denote by $\Omega[T]=\Hom_\Omega(-, T)$ the dendroidal set represented by a tree $T$ and by $\eta$ the representable $\Omega[L_0]$.

By the general arguments of left Kan extensions along the Yoneda embedding the inclusion $\Omega \to  \Oper$ induces an adjunction
\begin{equation*}
	\tau_d: \xymatrix{\dSet \ar@<0.3ex>[r] & \Oper: N_d \ar@<0.7ex>[l]}.
\end{equation*}
We call $N_d$ the \emph{dendroidal nerve functor} and for every coloured operad $P$ we have $N_d(P)_T = \Hom_{\Oper} (\Omega(T), P)$. Functor $N_d$ is fully faithful (as follows from \cite{MoerBar}, Proposition 7.1.4 and Proposition 7.3.7 and Proposition 7.3.8).

\begin{remark}
There is a fully faithful functor $i\colon \Delta \to \Omega$ sending the linear order $[n]$ to the linear tree $L_n$. It induces an adjunction
\begin{equation*}
i_!: \xymatrix{\sSet \ar@<0.3ex>[r] & \dSet: i^* \ar@<0.7ex>[l]}.
\end{equation*}
The functor $i_!$ is fully faithful. This and other good properties of this adjunction make dendroidal sets a generalization of simplicial sets.
\end{remark}

\begin{remark}
The inclusion of the full subcategory $\Omega^\circ$ on open trees into $\Omega$ also induces an embedding of the category of \emph{open dendroidal sets} (presheaves on $\Omega^\circ$) into the category of dendroidal sets. Where there is no danger of confusion we will consider open dendroidal sets as dendroidal sets. 
\end{remark}

\section{Elementary face maps} \label{elementary faces}

\subsection{Description of elementary face maps}

There are three types of elementary face maps : \emph{inner}, \emph{top} and \emph{bottom}. 

Let $e$ be an inner edge of a tree $T$. We define $\partial_e T$ to be the tree whose underlying set is $T\setminus \{e\}$, the partial order is induced from the one on $T$ and the set of leaves is the same as of $T$. There is an \emph{inner elementary face map} $\partial_e\colon \partial_e T \to T$ which is an inclusion of partially ordered sets. Note that if $e$ is an input of a vertex $v$ and the output of a vertex $w$, the tree $\partial_e T$ has a vertex $v\circ_e w := w \cup v \setminus \{e\}$ instead of $v$ and $w$. In terms of graphs, we obtain $\partial_e T$ by contracting the edge $e$:
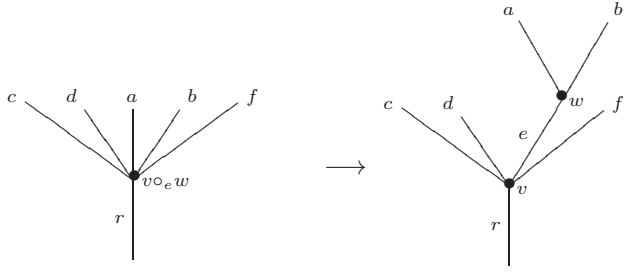
\begin{equation*}
	\xymatrix@C=12pt{
		&&&&&\\
		_c & _d & _a &_b&_f& \\
		&&*=0{\quad \quad \bullet_{v \circ_e w}} \ar@{-}[ull] \ar@{-}[ul] \ar@{-}[u] \ar@{-}[ur] \ar@{-}[urr]   &&&\\
		&&*=0{}\ar@{-}[u]^r &&&
	}
	\xymatrix{ \\ \\ \longrightarrow \\ }
	\xymatrix@C=12pt{
		&& _{a}&&_{b} &&& \\
		 _c & _d && *=0{\,\,\, \bullet_w} \ar@{-}[ul]  \ar@{-}[ur] &_f& \\
		&&*=0{\,\,\, \bullet_v} \ar@{-}[ull] \ar@{-}[ul] \ar@{-}[ur]^e \ar@{-}[urr]    &&&\\
		&&*=0{}\ar@{-}[u]^r &&&
	}
\end{equation*}

Let $w$ be a top vertex of a tree $T$.  We define $\partial_w T$ to be the tree whose underlying set is $T\setminus w$, the partial order is induced from the one on $T$ and the set of leaves is obtained by deleting the inputs of $w$ and adding the output of $w$ to the set of leaves of $T$. There is a \emph{top elementary face map} $\partial_w\colon \partial_w T \to T$ which is an inclusion of partially ordered sets. Note that if $T$ is a corolla with the root $r$ there is a unique top elementary face map and $\partial_w T$ is the unit tree with the unique edge $r$. In terms of graphs, we chop off the vertex $w$ and its inputs:
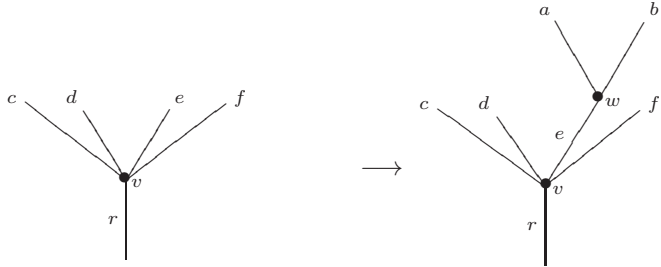
\begin{equation*}
	\xymatrix@C=12pt{
		&&&&&& \\
		 _c & _d &&_e&_f \\
		&&*=0{\,\, \bullet_v} \ar@{-}[ull] \ar@{-}[ul] \ar@{-}[ur] \ar@{-}[urr]   &&\\
		&&*=0{}\ar@{-}[u]^r &&
	}
	\xymatrix{ \\ \\ \longrightarrow \\ }
	\xymatrix@C=12pt{
		&& _{a}&&_{b} && \\
		 _c & _d && *=0{\,\,\, \bullet_w} \ar@{-}[ul]  \ar@{-}[ur] &_f \\
		&&*=0{\,\,\, \bullet_v} \ar@{-}[ull] \ar@{-}[ul] \ar@{-}[ur]^e \ar@{-}[urr]    &&\\
		&&*=0{}\ar@{-}[u]^r &&
	}
\end{equation*}

Let $v$ be a bottom vertex of a tree $T$ and $e$ an input of $v$ such that all other inputs of $v$ are leaves. If $T$ has at least two vertices, then $e$ is an inner edge and it is a unique input of $v$. 
We define the tree $\partial_vT$ with the underlying set $\{t \in T: e \leq t\}$, the induced partial order from $T$ and the set of leaves obtained by deleting the siblings of $e$ from the set of leaves of $T$. There is a bottom elementary face map $\partial_v \colon \partial_v T \to T$ which is an inclusion of partially ordered sets. In terms of graphs, we chop off $v$ with the root and all inputs of $v$ different from $e$:
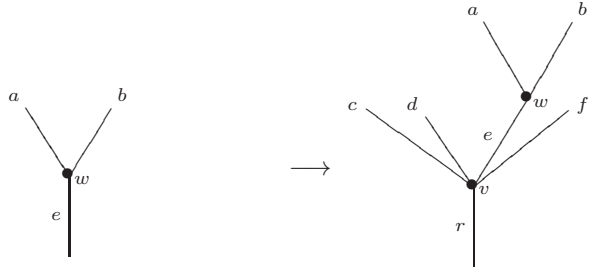
\begin{equation*}
	\xymatrix@C=12pt{
		& &&&&& \\
		& _a & &_b&  \\
		&&*=0{\,\,\, \bullet_w} \ar@{-}[ul]  \ar@{-}[ur]   &&\\
		&&*=0{}\ar@{-}[u]^e &&
	}
	\xymatrix{ \\ \\ \longrightarrow \\ }
	\xymatrix@C=12pt{
		&& _{a}&&_{b} &&& \\
		 _c & _d && *=0{\,\,\, \bullet_w} \ar@{-}[ul]  \ar@{-}[ur] &_f& \\
		&&*=0{\,\, \bullet_v} \ar@{-}[ull] \ar@{-}[ul] \ar@{-}[ur]^e \ar@{-}[urr]    &&&\\
		&&*=0{}\ar@{-}[u]^r &&&
	}
\end{equation*}

In the special case when $T$ is a corolla $C_n$, we have we have $n$ bottom elementary face maps
\[
\partial_{v,e} \colon \eta_e \to C_n,
\] one for each input $e$ of a unique vertex $v$.

\begin{remark} 
	We will usually write $\partial_f$ for a generic elementary face map and $f$ will denote either an inner edge or an outer vertex. 
\end{remark}

\begin{remark}
	If $\partial_f T \to T$ is an elementary face map, then every operation of $\partial_fT$ is also an operation of $T$, hence by Lemma \ref{operations_map} elementary face maps are morphisms of operads. In fact, the elementary face maps are monomorphisms in $\Omega$.
\end{remark}

\begin{prop}
	Every monomorphism  in $\Omega$ can be decomposed as a composition of elementary face maps.
\end{prop}

\begin{proof}
	The statement has been stated in \cite{den set} as Lemma 3.1 and discussed in \cite{MoerBar} as Lemma 2.3.2. 
\end{proof}

\begin{definition}
	Let $S$ be a simple face of $T$. For an inner edge $e$ of $T$ we say that $\partial_e S$  \emph{exists} if $e$ is also an inner edge of $S$. Analogously, we say that $\partial_w S$ or $\partial_v S$ exist if $w$ is a top vertex or $v$ is the bottom vertex (with all inputs except possibly one being a leaf) of $S$.
\end{definition}

\subsection{Relations between face maps.} 

When working with simplicial sets, one usually considers the skeleton category $\Delta$ of non-empty finite linear orders and writes $\partial_j \colon \{0,1\ldots, n-1\}\to \{0,1,\ldots, n\}$ for the unique non-decreasing monomorphism (elementary face map) that omits $j$ in the image. With this notation, the relation between these elementary face maps states: 
\[
\partial_i \partial_j = \partial_{j-1}\partial _i, \quad \text{for } i<j.   
\] 
If we instead consider the category of all finite linear orders and consider simple face maps as $\partial_{a} \colon T\setminus \{a\} \to T$ missing $a$ in the domain, the relation would read 
\[ \partial_a\partial_b=\partial_b\partial_a, \quad \text{for any } a \text { and } b.\]

Dendroidal elementary face maps and their relations are similar, but the situation is more complicated since there are exemptions to the described relation. The difference between the trees in the domain and the codomain of an elementary face map is not only one edge, but a set of edges and for different domains these sets might have non-trivial intersection. 

\begin{definition} 
	Let $T$ be a tree with at least two vertices. 
	Let $v$ be an outer vertex and $e$ the unique inner edge of a tree $T$ attached to $v$. We say that a pair $\{\partial_v, \partial_e\}$ is a \emph{mixed pair of elementary face maps of $T$}.
\end{definition}

\begin{prop}
	Let $\{\partial_f, \partial_g\}$ be a pair of elementary face maps of a tree $T$ with at least two vertices, which is not mixed. There are elementary face maps 
	\[
	\partial_f \colon \partial_f \partial_g T \to \partial_g T \quad \text{and} \quad \partial_g \colon \partial_g \partial_f T \to \partial_f T,
	\]
	the trees $\partial_f \partial_g T$ and $\partial_g \partial_f T$ are the same and the following dendroidal relation holds:   
	\[
	\partial_f \partial_g = \partial_g \partial_f. 
	\]
\end{prop}

\begin{proof}
  The statement may be easily checked by the reader. It has been  described in detail in Section 4 of \cite{GLW} and Section 2.2.3 of \cite{MoerBar}. 
\end{proof}

Let us consider a mixed pair $\{\partial_v, \partial_e\}$ elementary faces of a tree $T$, where $v$ is a top vertex. Let $e$ be attached to another vertex $w$. 
The elementary face \[ \partial_w \colon \partial_w\partial_v T \to \partial_v T \]
exists if and only if the elementary face \[ \partial_{w\circ_e v} \colon \partial_{w\circ_e v} \partial_e T \to \partial_e T \]
exists. This is the case if all inputs of $w$ other than $e$ are leaves of $T$ and the following dendroidal relation holds:
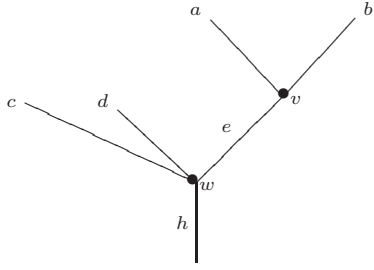
\[
	\partial_w \partial_v =\partial_{w\circ_e v} \partial_e. 	
\]

\[ 
\xymatrix{
	&&& _{a}&&_{b} &&&& \\
	& _c & _d && *=0{\,\,\, \bullet_v} \ar@{-}[ul]  \ar@{-}[ur] &&& \\
	&&&*=0{\,\, \bullet_w} \ar@{-}[ull] \ar@{-}[ul] \ar@{-}[ur]^e     &&&&\\
	&&&*=0{}\ar@{-}[u]^h &&&& 
}
\]

\begin{definition}
	Let $e$ be an inner edge attached to a top vertex $v$ and another vertex $w$ such that all other inputs of $w$ are leaves. 
	We say that a pair of the form  $\{\partial_v, \partial_w\} $ or of the form $\{\partial_e,\partial_{w\circ_e v} \}$ is an \emph{adjacent pair of elementary face maps}. 
	
	Analogously, we define adjacent pairs $\{\partial_v, \partial_w\}$, $\{\partial_e,\partial_{v\circ_e w}\}$ for the bottom vertex $v$ and the unique inner edge $e$ attached to $v$ and $w$ (all inputs of $w$ are leaves).
\end{definition}

\begin{remark}
For completeness, we describe the general case of a mixed pair of elementary face maps and summarize the discussion about the dendroidal relations. 

Let us discuss the case where an inner edge $e$ is attached to a top vertex $v$ and another vertex $w$. Let $h$ be the output of the vertex $w$. There is a unique maximal subtree $S$  of $T$ for which $h$ is a leaf. It is obtained by deleting all edges and vertices above $h$. This can certainly be achieved by first contracting the edge $e$ and then chopping off top vertices in a certain order ending in chopping off vertex $w\circ_e v$. Another way to obtain $S$ from $T$ is by chopping off top vertices starting with $v$ and ending with $w$. There are certainly more ways to obtain $S$ as we may chop off vertices in different order. There are also other maximal subtrees of $T$ that are contained in the intersection of $\partial_v T$ and $\partial_e T$ - one for each input of $w$. 

Of course, similar consideration holds when $v$ is the bottom vertex, as we may commonly think of leaves and the root as outer edges of the tree. Since the choice of a subtree in the intersection is not canonical in any way relevant to further considerations, we leave the conclusion in the following form. 
\end{remark}

\begin{prop}
Let $T$ be a tree with at least two vertices. For any pair \[\partial_f \colon \partial_f T \to T \text{ and } \partial_g \colon \partial_gT \to T\] of elementary face maps, there are elementary face maps
$\partial_{f_1},..., \partial_{f_r}$ and $\partial_{g_1},..., \partial_{g_r}$ such that \[ \partial_{f_r}...\partial_{f_1} \partial_g T= \partial_{g_r}...\partial_{g_1} \partial_f T.\]	
\end{prop}

\subsection{The partially ordered set of faces of a tree}

\begin{definition}
	We say that a tree $S$ is a \emph{face} of a tree $T$ if there is a sequence of elementary face maps $\partial_{f_1},..., \partial_{f_r}$ such that $S=\partial_{f_r}...\partial_{f_1}T$. We also say that $T$ is an \emph{extension} of $S$. We say that $\partial_{f_1},..., \partial_{f_r}$ is an \emph{extension sequence of $S$ to $T$}. 
\end{definition}

\begin{remark}
	Faces of a tree $T$ are representatives of subobjects of $T$ in $\Omega$. Also, if $S$ is a face of $T$, then $\Omega[S]$ is a representative of a subobject of $\Omega[T]$ in \dSet. 
\end{remark}

\begin{definition} 
	We denote by $\Sub(T)$ the family of all faces of a tree $T$. 
\end{definition}

For a tree $T$, the family $\Sub(T)$ is partially ordered by the relation of being a face. This poset is graded with the rank function given by the number of vertices. 

\begin{definition} \label{remF4}
	A pair $\{\partial_f, \partial_g\}$ of elementary face maps which are extensions of a tree $S$ is \emph{bad} if $f$ and $g$ are both top vertices or both bottom vertices attached to the same unique inner edge. Otherwise, we say that the pair $\{\partial_f, \partial_g\}$ is good. 
\end{definition}

\begin{lemma} \label{lemmaAxioms345} 
	Let $T$ be a tree. Consider faces $P$, $P_1$ and $P_2$ of $T$  and elementary face maps $\partial_f \colon P\to P_1$ and $\partial_g \colon P \to P_2$. Let $\mathcal{S}$ be the set of all faces $S$ of $T$ for which there exist a positive integer $r$ and elementary face maps $\partial_{f_1},..., \partial_{f_r}, \partial_{g_1},... \partial_{g_r}$ such that \[P_1=\partial_{g_1} \partial_{g_2} \ldots \partial_{g_r} S, \quad P_2=\partial_{f_1} \partial_{f_2} \ldots \partial_{f_r} S.\]
	
	\noindent Then, the set $\mathcal{S}$ is non-empty and has a unique minimal element $P_1 \cup P_2$  with respect to the induced partial order from $\Sub(T)$.
\end{lemma}

\begin{proof}
	Since $P_1$ and $P_2$ are faces of $T$ (and the face lattice of $T$ is graded), $T$ itself is an element of $\mathcal{S}$.
	Since $\mathcal{S}$ is finite, it has minimal elements. Assume $S_1$ and $S_2$ are two different minimal elements of $\mathcal{S}$. Their intersection (as dendroidal sets) is a disjoint union of faces of $T$. Also, the intersection contains $P_1$ and $P_2$. Every connected component of the intersection is a face of $S_1$ and a face of $S_2$. Since $P$ is connected there is a unique tree $S$ in the intersection of $S_1$ and $S_2$ which contains $P$. Both $P_1$ and $P_2$ are connected and have non-empty intersection with $P$, so they are also contained in $S$. Note that $S$ is a face of both $S_1$ and $S_2$. Also, $P_1$ and $P_2$ are faces of $S$, so $S$ is an element of $\mathcal{S}$. This contradicts the minimality of $S_1$ and $S_2$. Hence, $\mathcal{S}$ has a unique minimal element. 
\end{proof}

\begin{example}
	Let $P$ be a tree with one edge $e$ and no vertices, and let $P_1$ and $P_2$ be the trees as in the following picture. 
	
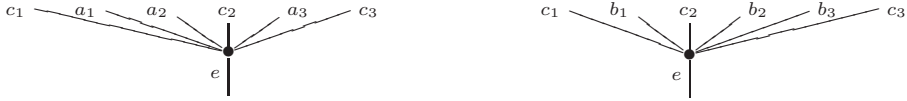
\[  
	\xymatrix@R=10pt@C=12pt{
		_{c_1} &_{a_1}& _{a_2} & _{c_2} &  _{a_3} & _{c_3}  && \\
		&&&*{ \bullet}  \ar@{-}[ulll] \ar@{-}[ull] \ar@{-}[ul] \ar@{-}[u] \ar@{-}[ur] \ar@{-}[urr]  &&&&\\
		&&&*=0{}\ar@{-}[u]^{e} &&&&
	}
	\xymatrix@R=10pt@C=12pt{
		& _{c_1} &_{b_1}&  _{c_2} & _{b_2} &  _{b_3} & _{c_3}  && \\
		&&&*{ \bullet}   \ar@{-}[ull] \ar@{-}[ul] \ar@{-}[u] \ar@{-}[ur] \ar@{-}[urr] \ar@{-}[urrr] &&&&\\
		&&&*=0{}\ar@{-}[u]^{e} &&&&
	}
	\]	
	Then $P_1\cup P_2$ is the following tree (and we may think $T$ is the same tree). 
	
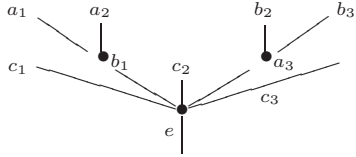
\[  
	\xymatrix@R=10pt@C=12pt{
		& _{a_1}&_{a_2} &&_{b_2}&_{b_3}&&&& \\
		&_{c_1}& *{ \quad \bullet_{b_1}} \ar@{-}[ul]  \ar@{-}[u] & _{c_2} & *{\quad  \bullet_{a_3}} \ar@{-}[u]  \ar@{-}[ur] &  && \\
		&&&*{ \bullet}  \ar@{-}[ull]\ar@{-}[ul] \ar@{-}[u] \ar@{-}[ur] \ar@{-}[urr]_{c_3}   &&&&\\
		&&&*=0{}\ar@{-}[u]^{e} &&&&
	}
	\]	
\end{example}

\begin{remark} \label{extension sequence remark}
Given elementary face maps $\partial_f \colon P\to P_1$ and $\partial_g \colon P \to P_2$, we can explicitly construct $P_1\cup P_2$. Note that for a good pair $\{\partial_f, \partial_g\}$ the construction of $P_1\cup P_2$ is obvious and we have $r=1$, $f_1=f$ and $g_1=g$. 

We describe the construction in the case when $f$ and $g$ are top vertices attached to the same edge $e$. Let us write $f\cap g=C$, $f=A\cup C$ and $g=B\cup C$, with $A\cap B=\emptyset$. Furthermore let us enumerate 
\[ C=\{c_1,\ldots, c_k\}, \quad A=\{a_1,\ldots, a_n\}, \quad B=\{b_1, \ldots, b_m\} \]
in such way that there are partitions
\[
A=A_1\cup \ldots \cup A_r, \quad B=B_1\cup \ldots \cup B_r,
\]
and $s\in \{1, \ldots, r\}$ satisfying \[B_i=\{b_i\} \text{ and } b_i \leqslant a \text{ for } a\in A_i, \, i=1,\ldots,s, \] while \[A_i=\{a_{m+r-i}\} \text{ and }  a_{m-r+i} \leqslant b \text{ for } b\in B_i,\, i=s+1, \ldots, r.\] 
Then $P_1 \cup P_2$ is constructed so that \[ C\cup B_1\cup \ldots \cup B_s \cup A_{s+1} \cup \ldots \cup A_r \] is the set of inputs of the vertex with output $e$, $A_i$ is the set of inputs of the vertex with output $b_i$ for $i=1, \ldots, s$, and $B_i$ is the set of inputs of the vertex $a_{m+r-i}$ for $i=s+1, \ldots, r$. 
\end{remark}

\subsection{Planar structures of trees and the total order of elementary faces} \label{SectOrder}

\begin{definition}
A \emph{planar structure} on a tree $T$ is a family of total orders $(v,\preceq_v)$, one for each vertex $v$ of $T$.
\end{definition}

\begin{lemma}
Let $(T, \leq, L)$ be a tree and $e,f$ two distinct elements of $T$ other than the root $r_T$. There exist unique siblings $e', f' \in T$ such that $e'\leq e$ and  $f'\leq f$.
\end{lemma}

\begin{proof}
If $f\leq e$ then $e'=f'=f$. Similarly, if $e\leq f$ then $e'=f'=e$. Otherwise, let us assume that $e$ and $f$ are not comparable. The finite totally ordered set $\{h\in T \mid h\leq e\} \cap \{h\in T \mid h\leq f\}$ is non-empty since it contains the root of $T$. Hence there exists a largest element $g$ such that $g\leq e$ and $g \leq f$. Then $e'$ is the smallest element such that $g < e' \leq e$ and $f'$ is the smallest element such that $g < f' \leq f$. By minimality $e'$ and $f'$ are immediate predecessors of $g$ and hence siblings.
\end{proof}

For every planar structure on $(T, \leq, L)$ given by a family of total orders $\preceq_v$ on vertices $v$, we can define a relation $\preceq$ on the set $T$ by
\begin{equation}\label{NatTotalOrd}
e \preceq f \Leftrightarrow e' \preceq_v f'
\end{equation} for $e'$ and $f'$ associated to $e$ and $f$ by the previous lemma. One easily checks that $\preceq$ is a total order on $T$ which extends the partial order $\leq$.

\begin{example}
In terms of graphs this total order is given by traversing the tree $T$ from left to right and from bottom to top. We have
$\{r \preceq c \preceq d \preceq e\preceq  a\preceq b\preceq f\}$ for the planar tree in Example \ref{tree}.
\end{example}

Every total order $\preceq$ extending the partial order $\leq$ of a tree $T$ induces a total order
$\leq$ on the set of operations of $T$ such that $(A, t) \leq (B, s)$
 if
 \begin{itemize}
 \item $t\preceq s$ or
 \item $t=s$ and $A$ is empty or
 \item $t=s$, $A=\{t_1,...,t_m\}$, $B=\{s_1,...,s_n\}$ and there is a positive integer $k$ such that $t_i=s_i$ for for $1\leq i \leq k-1$ and $t_k \preceq s_k$, $t_k\neq s_k$.
\end{itemize}

To every elementary face map of $T$ we can assign an operation of $T$ - we assign $(e;e)$ to an inner elementary face map $\partial_e$, $(w;e)$ to a top elementary face map $\partial_w$ where $e$ is the output of $w$, and $(v;r)$ to a bottom face map $\partial_{v}$. 

For any face $S$ of $T$ with at least two vertices, this gives a total order on the set of faces of $S$
\begin{equation*}
\F(S) = \{\partial_f \mid \partial_f : \partial_f S \to S\}
\end{equation*}
because to any elementary face map $\partial_f$ we associated an operation in $S$ which is also an operation in $T$. The case of a corolla $S$ is an exception as we have assigned the same operation to all bottom faces of a corolla, but that is not relevant because we will use this total order on $\F(S)$ only when $S$ has at least two vertices. Of course, the total order of edges of a corolla $S$ gives a natural total order on $\F(S)$, which we will not need.

Also, we get a total order on the set of extensions of $S$
\[
\E(S) = \{\partial_f \mid \partial_f : \partial_f R \to R, R \in \Sub(T), S=\partial_f R \}
\]
because to any elementary face map $\partial_f$ we associated an operation in $R$ which is also an operation in $T$. These considerations have the following important consequence that we will use in the next section. 

\begin{cor} \label{OrderComp}
For faces $S$ and $R$ of a tree $T$ and a commutative square
$$ \xymatrix{S \ar[d]_{\partial_f} \ar[rr]^{\partial_g} && \partial_f R \ar[d]^{\partial_f} \\  \partial_gR \ar[rr]_{\partial_g} && R} $$
of maps in $F$ with $S= \partial_f \partial_g R = \partial_g \partial_f R$, we have $\partial_f \leq \partial_g$ in $\E(S)$ if and only if $\partial_f \leq \partial_g$ in $\F(R)$.
More generally, if $P$ and $Q$ are faces of $T$ such that $\partial_f, \partial_g \in \F(P)$ and $\partial_f, \partial_g \in \F(Q)$ then $\partial_f \leq \partial_g$ in $\F(P)$ if and only if $\partial_f \leq \partial_g$ in $\F(Q)$. Analogous statement holds for the sets of extensions $\E(P)$ and $\E(Q)$. 
\end{cor}

\section{The combinatorics of dendroidal anodyne extensions}

\subsection{The method of canonical extensions.} 

\begin{definition} A subobject of a tree $T$ in the category $\Omega$ is represented by a face $S$. In the category of dendroidal sets, a subobject $A$ of a representable $\Omega[T]$ is a union of representables represented by a set of faces of $T$. We will often describe subobjects $A\subseteq \Omega[T]$ equivalently by specifying those faces $P$ of $T$ that do not factor through the inclusion $A\to \Omega[T]$ and we will call such face $P$ a \emph{missing face with respect to $A$}. In that case we write $P\not \subseteq A$.
\end{definition}

\begin{definition}
	Any elementary face map $\partial_f : \partial_f T \to T$ induces a map of representable dendroidal sets $\partial_f : \Omega[\partial_f T] \to \Omega[T]$. The union of all images of maps $\partial_f : \Omega[\partial_f T] \to \Omega[T]$ is denoted by $\partial \Omega[T]$.
	There is an inclusion $\partial \Omega[T] \to \Omega[T]$ and any such map is called a \emph{boundary inclusion}.
\end{definition}

\begin{definition}
The smallest class closed under pushouts, retracts and transfinite compositions containing all boundary inclusions is called the class of \emph{normal monomorphisms}. 
\end{definition}

\begin{definition}
	For an elementary face map $\partial_f : \partial_f T \to T$ we denote by $\Lambda^f[T]$ the union of images of all elementary face maps $\partial_g : \Omega[\partial_g T] \to \Omega[T], g\neq f$.
	
	There is an inclusion $\Lambda^f[T] \to \Omega[T]$ and any such map is called a \emph{horn inclusion}. A horn inclusion is called inner (resp. top or bottom) if $\partial_f$ is an inner (resp. top or bottom) elementary face map. 
\end{definition}

\begin{definition}
The smallest class of normal monomorphisms that is closed under pushouts, retracts and transfinite compositions containing inner (resp. inner and top, all) horn inclusions is called the class of \emph{operadic anodyne extensions} (resp. \emph{covariant anodyne extensions}, \emph{stable anodyne extensions}). 
\end{definition}

Let $R$ be a tree. Under certain conditions on a dendroidal subset $A$ of the representable dendroidal set $\Omega[R]$ we will show that the inclusion $A\to \Omega[R]$ is a dendroidal anodyne extension (operadic, covariant or stable). The approach that we will present has the advantage of being applicable to obtain many old and new results and that these conditions on $A$ are easily verified in the concrete cases that we consider. The idea is to form a filtration
\begin{equation*} \label{filtration}
A=A_0 \subset A_1 \subset ...\subset A_{N-1} \subset A_N =\Omega[R]
\end{equation*}
in which every inclusion $A_n \to A_{n+1}$ is a pushout of a coproduct of a family of horn inclusions of faces of $R$, i.e. fits into a pushout diagram
\begin{equation*}
\xymatrix{\coprod \Lambda^f [P]    \ar[d] \ar[r] & A_n \ar[d] \\  \coprod \Omega[P] \ar[r] & A_{n+1}}
\end{equation*}
where the coproduct ranges over pairs $(\partial_f P, P)$ of faces of $R$ that will be carefully formed and ordered in the way we now describe in detail.

\begin{definition}
	Let $R$ be a tree and $A \subseteq \Omega[R]$. Let $F$ be a subset of the set 
	\[ \{\partial_f \colon \partial_f P \to P \mid P \in \Sub(R); \quad  P, \partial_f P\not\subseteq A\}. \]
	For every missing face $P$ of $R$ we define the set of \emph{$F$--extensions} of $P$
	\begin{equation*}
		\E_F(P) =\{ \partial_f \colon P=\partial_fP' \to P' \mid \partial_f \in F\}
	\end{equation*}
	and the set of \emph{$F$--faces} of $P$
	\begin{equation*}
		\F_F(P)=\{\partial_f\colon \partial_f P\to P \mid \partial_f \in F \}.
	\end{equation*}
\end{definition}

\begin{definition} \label{def extension set}
Let $R$ be a tree and $A \subseteq \Omega[R]$.  We say that a subset $F$ of the set 
\[ \{\partial_f \colon \partial_f P \to P \mid P \in \Sub(R); \quad  P, \partial_f P\not\subseteq A\}
\]  is an \emph{extension set} with respect to $A$ if the following Axioms (F1)-(F5) are satisfied.\\

\begin{enumerate}
	\item[(F1)] \textbf{The Forbidden Pair Axiom.} The set $F$ does not contain any mixed, adjacent or bad pair of elementary face maps.
	\item[(F2)] \textbf{The Bad Pair Axiom.} For any extension $\partial_g \colon P\to P'$ which is not an element of $F$ there is at most one extension $\partial_f\colon P \to P''$ in $F$ such that $\{\partial_f, \partial_g\}$ is a bad pair.
	\item[(F3)] \textbf{The Face Closure Axiom.} For any commutative square  \[ \xymatrix{\partial_g\partial_f P \ar[d]_{\partial_f} \ar[rr]^{\partial_g} && \partial_f P \ar[d]^{\partial_f} \\  \partial_gP \ar[rr]_{\partial_g} && P,} \] if any two maps labeled $\partial_f$ and $\partial_g$ are in $F$, then all four maps are in $F$.
	\item[(F4)] \textbf{The Extension Closure Axiom.} For any extension $\partial_f\colon P \to P'$ in $F$ and any  extension $\partial_g \colon P\to P''$ (not necessarily in $F$), all elements of any extension sequence $\partial_{f_1}, ..., \partial_{f_r}$ of $P''$ to $P'\cup P''$ are elements of $F$.
	\item[(F5)]	\textbf{The Existence Axiom.} For any missing face $P$, at least one of the sets $\F_F(P)$ and $\E_F(P)$ is non-empty.
\end{enumerate}
\end{definition}

\begin{example} \label{SegalCore}
	The \emph{Segal core} $Sc[R]$ of a tree $R$ is the union of images of all monomorphisms $\Omega[C_n] \to \Omega[R]$ which are compositions of only top and bottom elementary face maps (no inner elementary face maps). 
	The set 
	\begin{equation*}
		F=\{\partial_e \colon \partial_e P \to P \mid P\in \Sub(R) , e \textrm{ is an inner edge of } P  \}
	\end{equation*} is an extension set with respect to $Sc[R]$. The axioms (F1), (F2), (F3) and (F4) follow because $F$ contains only inner face maps. The Existence Axiom is obvious as each missing face $P$ either has an inner edge (so $\F_F(P)$ is non-empty) or it is a corolla obtained by contracting an inner edge in a bigger tree (so $\E_F(P)$ is non-empty). Theorem \ref{theorem method} will give one more proof that the inclusion $Sc[R] \to \Omega[R]$ is an operadic anodyne map, originally proven in \cite{den Seg sp} as Proposition 2.4.
\end{example}

For the rest of this section let us fix a tree $R$, a planar structure on $R$, a dendroidal subset $A$ and an extension set $F$ with respect to $A$. By the considerations in subsection \ref{SectOrder}, the planar structure on $R$ induces a total order on every set $\E_F(P)\subseteq \E(P)$ and $\F_F(P)\subseteq \F(P)$ for every face $P$ of $R$. By Corollary \ref{OrderComp} these total orders are compatible in the sense that for two elementary face maps $\partial_f$ and $\partial_g$, if $\partial_f \leq \partial_g$ in any set $\E_F(P)$ or $\F_F(P)$ for some $P$, then the same relation holds in all sets $\E_F(P)$ and $\F_F(P)$ that contain both $\partial_f$ and $\partial_g$.

\begin{definition} Let $R$ be a planar tree, $A \subseteq \Omega[R]$ and $F$ an extension set with respect to $A$. 
Let $P$ be a face of $R$ such that $\F_F(P)$ is non-empty. We say that an elementary face map $\partial_f \colon \partial_f P \to P$ is \emph{a canonical extension} if $\partial_f=\min \F_F(P)$ and $\partial_f=\min \E_F(\partial_f P)$. Since an elementary face map is determined by its domain and codomain we also say that the pair $(\partial_f P, P)$ is a canonical extension.
\end{definition}

\begin{remark}
It might happen that any one of the conditions $\partial_f=\min \F_F(P)$ and $\partial_f = \min \E_F(\partial_f P)$ holds, while the other does not hold.
\end{remark}

\begin{lemma}  \label{disjoint can ext}
Canonical extensions are disjoint. More precisely, for an extension set $F$, the following two statements hold.    
	
\begin{enumerate} 
	 \item \label{lema1}  For any two canonical extensions $(\partial_{f_1} P_1, P_1)$ and $(\partial_{f_2} P_2, P_2)$, $P_1=P_2$ holds if and only if $\partial_{f_1} P_1 = \partial_{f_2} P_2$ holds.
	\item \label{lema2} Pairs $(\partial_g \partial_f P, \partial_f P)$ and $(\partial_f P, P)$ can not be both canonical extensions.
\end{enumerate}

\end{lemma}

\begin{proof}
\begin{enumerate}
	
\item The statement follows from the fact that minimal faces and minimal extensions are unique. 

\item Since the set $F$ does not contain an adjacent pair of elementary face maps, there is a tree $\partial_g P$ and the commutative square 

\[ \xymatrix{\partial_g\partial_f P \ar[d]_{\partial_f} \ar[rr]^{\partial_g} && \partial_f P \ar[d]^{\partial_f} \\  \partial_gP \ar[rr]_{\partial_g} && P.} \]

By Remark \ref{OrderComp} the total orders on $\F_F(P)$ and $\E_F(\partial_g\partial_gP)$ are compatible, so it is impossible that $\partial_f$ is the least element of 
$\F_F(P)$ and $\partial_g$ is the least element of $\E_F(\partial_g\partial_gP)$ as this would mean $\partial_f \leq \partial_g $ and $\partial_g\leq \partial_f$. 
\end{enumerate}
\end{proof}

\begin{lemma} \label{lema3} Every missing face is the domain or the codomain of a canonical extension.	
\end{lemma}

\begin{proof}
Let $P$ be a missing face which is not a codomain and let us show that it is a domain of a canonical extension. First, we claim that $\E_F(P)\neq \emptyset$. 

If $\F_F(P)$ is empty, this is implied by The Existence Axiom. If $\F_F(P)$ is non-empty and $\partial_f =\min \F_F(P)$, then by the definition of canonical extensions, there exists an elementary face map $\partial_k \in \E_F(\partial_fP)$ such that $\partial_k < \partial_f$. Since $F$ does not contain a bad pair of extensions and it is closed under extensions, there is a commutative square of maps in $F$:

\[ \xymatrix{\partial_f P = \partial_k P' \ar[d]_{\partial_f} \ar[rr]^{\partial_k} &&  P' \ar[d]^{\partial_f} \\  P \ar[rr]_{\partial_k} && P'',} \]
which shows that $\E_F(P)$ is non-empty.    

Let $\partial_g=\min \E_F(P)$, $\partial_g: P\to P_1$.
We claim that $\partial_g=\min \F_F(P_1)$. First of all, since the set $F$ does not contain a mixed pair of face maps and it is closed under taking faces, we have the following commutative square of maps in $F$:  

\[ \xymatrix{\partial_f P \ar[d]_{\partial_f} \ar[rr]^{\partial_g} &&  \partial_fP_1 \ar[d]^{\partial_f} \\  P \ar[rr]_{\partial_g} && P_1} \]
and we conclude that $\partial_g \leqslant \partial_f$ holds. Furthermore, since the set $F$ does not contain an adjacent pair of face maps and it is closed under taking faces, for any map $\partial_h\in  \F_F({P_1})$ such that $\partial_h < \partial_g$ there would be a commutative square of maps in $F$ 
\[ \xymatrix{\partial_h P \ar[d]_{\partial_h} \ar[rr]^{\partial_g} &&  \partial_h P_1 \ar[d]^{\partial_h} \\  P \ar[rr]_{\partial_g} && P_1,} \]
which contradicts $\F_F(P)=\emptyset$ or $\partial_f=\min \F_F(P)$. Hence $\partial_g = \min \F_F(P_1)$ and $(P,P_1)$ is a canonical extension.
\end{proof}

\begin{lemma} \label{lema4}
Let $(\partial_f P, P)$ be a canonical extension. For any elementary face map $\partial_g : \partial_g P\to P$, with $g \neq f$, one of the following holds:
 \begin{itemize}
\item $\partial_g  P$ is not a missing face with respect to $A$;
\item $\partial_f\in \F_F(\partial_gP)$ and the pair $(\partial_f\partial_g P, \partial_g P)$ is a canonical extension;
\item $\card{\E_F(\partial_gP)}< \card{\E_F(\partial_f P)}$.
\end{itemize}
\end{lemma}

\begin{proof} Let us assume that $\partial_g P$ is a missing face with respect to $A$. If $\partial_g \colon \partial_g P \to P$ is an element of $F$, then there is a commutative square
\[ \xymatrix{\partial_g\partial_f P \ar[d]_{\partial_f} \ar[rr]^{\partial_g} &&  \partial_fP \ar[d]^{\partial_f} \\  \partial_gP \ar[rr]_{\partial_g} && P} \]
with $\partial_f \in \F_F(\partial_g P)$ because $F$ does not contain an adjacent pair of face maps and it is closed under taking faces.
For every $\partial_h \colon \partial_h \partial_g  P \to \partial_g  P$ in $F$, there is a commutative diagram 

\[ \xymatrix{\partial_g\partial_f P \ar[d]_{\partial_f} \ar[rr]^{\partial_g} &&  \partial_fP \ar[d]^{\partial_f} \\  \partial_gP \ar[rr]_{\partial_g} && P \\
\partial_h\partial_gP \ar[u]_{\partial_h} \ar[rr]_{\partial_g} && \partial_hP \ar[u]_{\partial_h}
} \]
 in which all maps are in $F$, because $F$ does not contain a mixed pair of faces and it is closed under faces. Since $(\partial_f P, P)$ is a canonical extension, we have $\partial_f \leqslant \partial_h$ and we conclude that $\partial_f$ is the least element of $\F_F(\partial_gP)$. Similarly, for every extension $\partial_k\colon \partial_g \partial_f P=\partial_k P' \to P'$ in $F$, there is a commutative diagram 

\[ \xymatrix{
	 P'  \ar[rr]^{\partial_g} && P''   \\
	\partial_g\partial_f P \ar[u]^{\partial_k} \ar[d]_{\partial_f} \ar[rr]^{\partial_g} &&  \partial_fP \ar[u]_{\partial_k} \ar[d]^{\partial_f}  \\  \partial_gP \ar[rr]_{\partial_g} && P } \]
in which all maps are in $F$, because $F$ does not contain a bad pair of extensions and it is closed under extensions. Since $(\partial_f P, P)$ is a canonical extension, we have $\partial_f \leqslant \partial_k$ and we conclude that $\partial_f$ is the least element of $\E_F(\partial_g\partial_fP)$. Hence, $(\partial_f \partial_g P, \partial_g P)$ is a canonical extension.

Otherwise, assume $\partial_g \colon \partial_g P \to P$ is not an element of $F$. For any extension \[\partial_k \colon \partial_g P =\partial_k P'  \to P', \quad  \partial_k \in F,\] there is a face $P''$ of $R$ and an elementary face map \[\partial_{k_1} \colon P=\partial_{k_1}P''\to P'',\] which is in $F$ because $F$ is closed under extensions. Let us choose one such map $\partial_{k_1}$ for every $\partial_k \in \E_F(\partial_g P)$ and denote this assignment \[\psi: \E_F(\partial_g P) \to \E_F(P).\] The Bad Pair Axiom implies that there is at most one element $\partial_k \in \E_F(\partial_g P)$ such that $\psi(\partial_k) \neq \partial_k$, so $\psi$ is injective and we conclude \[\card \E_F(\partial_g P) \leq \card{\E_F(P)}.\]
Since $F$ does not contain an adjacent pair of faces and it is closed under faces it follows that for every element $\partial_{k_1}$ in $\E_F(P)$, there is an extension $\partial_{k_1}$ in $\E_F (\partial_{f}P)$. Hence,  \[\card \E_F(P) \leq \card \E_F(\partial_{f}P). \]
Since $\partial_f \colon \partial_f P \to P$ and $\partial_g \colon \partial_g P \to P$ are elementary face maps of the same tree $P$, there is an edge of $P$ which appears in $\partial_gP$ but does not appear in $\partial_f P$. Hence $\partial_f$ is not an element of $\E_F(\partial_g P)$ and we have $\card \E_F(\partial_g P) < \card \E_F(\partial_f P)$.
\end{proof}

\begin{theorem} \label{theorem method}
 Let $R$ be a tree and $A$ a dendroidal subset of $\Omega [R]$ such that there exists an extension set $F$. Then, the inclusion $A \to \Omega [R]$ is a composition of pushouts of horns $\Lambda^f P \to P$ with $\partial_f \in F$.

 Hence, the inclusion $A \to \Omega [R]$ is a stable anodyne extension, which is moreover a covariant anodyne extension if all elements of $F$ are either inner or top elementary face maps and an operadic anodyne extension if all elements of $F$ are inner elementary face maps.
\end{theorem}

\begin{proof}
By Lemma \ref{lema3} every missing face of $R$ with respect to $A$ is either the first or the second component of a canonical extension $(P,P')$. By Lemma \ref{disjoint can ext}, all such pairs are mutually disjoint.

Let $\mathcal{P}_{n,c}$ be the family of all canonical extensions $(\partial_f P, P)$ such that $\partial_f P$ has $n$ vertices and $\card \E_F(\partial_f P)=c$. Let $m$ be the minimal number of vertices over all missing faces. Let $d$ be the minimal cardinality of the set $\E_F(P)$ over all missing faces $P$ with number of vertices being $m$. We define $A_{m,d}$ to be the union of the dendroidal set $A$ with the representables of all missing faces $P$ and their canonical extensions such that $P$ has $m$ vertices and $\card\E_F(P)=d$.
For notational convenience, we define $A_{n,c}=A_{m,d}$ if $1\leq n<m$ or if $n=m$ and $1\leq c<d$. We inductively define dendroidal sets $A_{n,c}$ as the union of
\begin{itemize}
\item all dendroidal sets $A_{n', c'}$ such that $n'< n$,
\item all dendroidal sets $A_{n', c'}$ such that $n'=n$ and $c' < c$, and
\item all representables $\Omega[P]$ and $\Omega[\partial_f P]$ such that $(\partial_f P, P)\in \mathcal{P}_{n,c}$.
\end{itemize}
For a fixed $n\geq 1$, if $c$ is the maximum of $\card \E_F(P)$ over all faces $P$ with $n$ vertices, we define $A_{n+1,0}=A_{n,c}$. Lemma \ref{lema4} implies that there is an inclusion
$$\coprod_{(\partial_f P, P) \in \mathcal{P}_{n,c}} \Lambda^f P \to A_{n,c-1}.$$
Since all canonical extensions are mutually disjoint, for any $(\partial_f P, P)\in \mathcal{P}_{n,c}$ the representable $\Omega[P]$ does not factor through $A_{n,c-1}$ so we have a pushout diagram $$\xymatrix{\ds \coprod_{(\partial_f P, P) \in \mathcal{P}_{n,c}} \Lambda^f [P]    \ar[d] \ar[rr] && A_{n,c-1} \ar[d] \\ \ds \coprod_{(\partial_f P, P) \in \mathcal{P}_{n,c}} \Omega[P] \ar[rr] && A_{n,c}.}$$
This proves the statement. 
\end{proof}

\section{Extension sets for shuffles of trees and the pushout-product property}

In this section we consider the tensor product of trees that yields a monoidal structure on the category of dendroidal sets. We will see that, as for linear orders in the theory of simplicial sets, the product of trees is a union of trees called \emph{shuffles}. We provide two examples of extension sets for faces of shuffles of two trees. These auxiliary results will be used in the last subsection multiple times to prove the pushout-product property for the stable model structure. 

\subsection{Tensor product of trees}

The category of (coloured) operads has a tensor product $\otimes_{BV}$, called the Boardman-Vogt tensor product, making it a closed symmetric monoidal category. This monoidal structure induces a tensor product on the category of dendroidal sets such that $$\Omega[S]\otimes \Omega[T] = N_d (\Omega(S) \otimes_{BV} \Omega(T))$$
for any two trees $S$ and $T$.  Details can be found in \cite{inn Kan in dSet}.
These tensor products are part of a colax symmetric monoidal structure, as described in \cite{Equivalence}. We do not go into details, as we use only binary tensor products in this article. 

The tensor product of two representables decomposes as a union of representables, called shuffles in this context. We repeat basic definitions and results needed for our applications and refer the reader for further details to Chapter 4 of \cite{MoerBar} and a more recent overview \cite{HoffMoer}. 

\begin{definition}  \label{def shuffle}
	Let $S$ and $T$ be trees. A \emph{shuffle of $S$ and $T$} is a tree $R$ such that:
	\begin{itemize}
		\item the set of edges of $R$ is a subset of $S \times T$,
		\item the root of $R$ is  $(r_S, r_T)$,
		\item the set of leaves $L(R)$ of $R$ is equal to the set $L(S)\times L(T)$.
		\item if $(s,t)$ is an edge of $R$ which is not a leaf, then either the inputs of the vertex above $(s,t)$ are of the form $(s_1, t)$, \ldots, $(s_m, t)$ where $s_1, \ldots, s_m$ are inputs of the vertex above $s$ in $S$, or these inputs are of the form $(s,t_1)$, \ldots, $(s, t_n)$ where $t_1, \ldots, t_n$ are inputs of the vertex above $t$ in $T$.
	\end{itemize}
\end{definition}

\begin{remark}
	We will call the vertices of the form $v\otimes t$ \emph{white} and the vertices of the form $s\otimes w$ \emph{black}. We will draw:
	$$
	
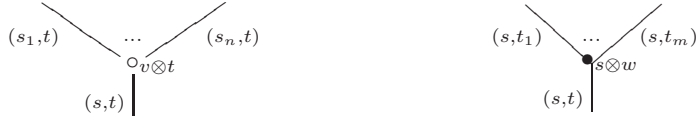
\xymatrix@R=3pt@C=8pt{
		&& &&&&& \\
		&& _{(s_1, t)} & _{...} &_{(s_n, t)}&  \\
		&&&*{\quad \, \, \circ_{v\otimes t}} \ar@{-}[uul]  \ar@{-}[uur]   &&\\
		&&&*=0{}\ar@{-}[u] && \\
		&&&*=0{}\ar@{-}[u]^{(s , t)} &&
	}
	\quad
	\xymatrix@R=3pt@C=8pt{
		&& &&&&& \\
		&& _{(s, t_1)} &_{...} &_{(s, t_m)}&  \\
		&&&*=0{\quad \, \, \bullet_{s\otimes w}} \ar@{-}[uul]  \ar@{-}[uur]   &&\\
		&&&*=0{}\ar@{-}[u]&& \\
		&&&*=0{}\ar@{-}[u]^{(s , t)} &&
	}
	$$
\end{remark}

\begin{remark}
As discussed in Section 2 of \cite{HoffMoer}, the fourth condition of Definition \ref{def shuffle} can be replaced by the condition: 
\begin{itemize}
\item for any two leaves $s$ of $S$ and $t$ of $T$, the branch from leaf $(s,t)$ in $R$ to the root of $R$ is a sequence of edges 
\[
(s,t)=(s_1, t_1), (s_2, t_2), \ldots, (s_k, t_k)=(r_S, r_T)
\] 
such that $s_i=s_{i-1}$ and $t_{i-1}$ and $t_i$ are consecutive edges in $T$ or $t_i=t_{i-1}$ and $s_{i-1}$ and $s_i$ are consecutive edges in $S$.

\end{itemize}
\end{remark}

\begin{example} \label{example shuffle}
	The following tree
	
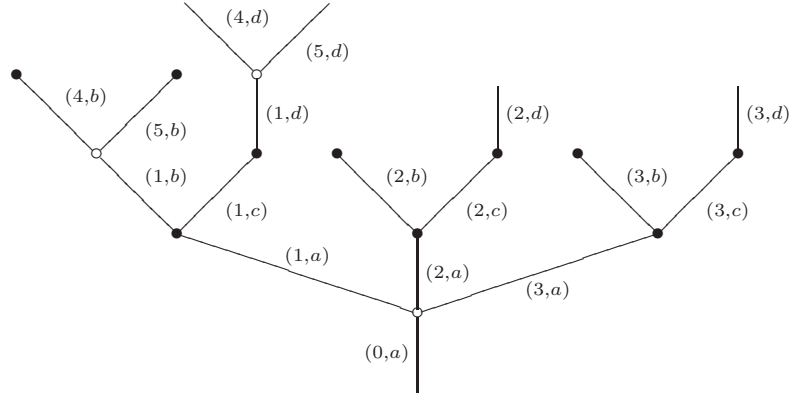
\[
	\xymatrix{
		&&&&&&&&& \\
		*=0{\bullet} && *=0{\bullet}   & *=<3.1pt>{\circ} \ar@{-}[ul]_{(4,d)} \ar@{-}[ur]_{(5,d)} &&&&&&& \\
		&*=<3.1pt>{\circ} \ar@{-}[ul]_{(4,b)} \ar@{-}[ur]_{(5,b)} &  & *=0{\bullet} \ar@{-}[u]_{(1,d)}  & *=0{\bullet}  && *=0{\bullet}  \ar@{-}[u]_{(2,d)}   &*=0{\bullet}  && *=0{\bullet}  \ar@{-}[u]_{(3,d)}   & \\
		&& *=0{\bullet} \ar@{-}[ul]_{(1,b)} \ar@{-}[ur]_{(1,c)} &&& *=0{\bullet} \ar@{-}[ul]_{(2,b)} \ar@{-}[ur]_{(2,c)} &&&  *=0{\bullet} \ar@{-}[ul]_{(3,b)} \ar@{-}[ur]_{(3,c)} &&  \\
		&&&&&*=<3.1pt>{\circ} \ar@{-}[ulll]_{(1,a)} \ar@{-}[u]_{(2,a)} \ar@{-}[urrr]_{(3,a)}  & &&\\
		&&&&&*=0{}\ar@{-}[u]^{(0,a)} &&
	}
	\]
	is an example of a shuffle of the trees
	\[
	\xymatrix@R=3pt@C=8pt{
		& & &&&&& \\
		& & &&&&& \\
		& *=<3.1pt>{\circ}\ar@{-}[ddrr]_1 *{}\ar@{-}[uul]^4 *{}\ar@{-}[uur]_5  &&   && & \\
		S= && &&&&& \\
		&&&*=<3.1pt>{\circ} *{}\ar@{-}[uu]_2 *{}\ar@{-}[uurr]_3   &&\\
		&&&*=0{}\ar@{-}[u]&& \\
		&&&*=0{}\ar@{-}[u]^0 &&
	}
	\xymatrix@R=3pt@C=8pt{
		&&  \\
		&&  \\
		&&  \\
		\text{ and } 	&&  \\
		&&  \\
		&&  \\
		&&  \\
	}	
	\xymatrix@R=3pt@C=8pt{
		&& &&&&& \\
		&& &&&&& \\
		&& *=0{\bullet} && *=0{\bullet} \ar@{-}[uu]_d &&  \\
		T= 	&& &&&&& \\
		&&&*=0{\bullet} \ar@{-}[uul]_b  \ar@{-}[uur]_c   &&\\
		&&&*=0{}\ar@{-}[u]&& \\
		&&&*=0{}\ar@{-}[u]^{a} &&
	}
	\]
\end{example}

\begin{prop}[\cite{inn Kan in dSet}, Lemma 9.5]
	Every shuffle $R$ of $S$ and $T$ comes with a canonical monomorphism $m:\Omega[R]\to \Omega[S]\otimes \Omega[T]$. If $R_i$, $i=1,...,N$ are all  shuffles of $S$ and $T$ then the dendroidal set $\Omega[S]\otimes \Omega[T]$ is isomorphic to the union of all $\Omega[R_i]$, i.e.
	$$\Omega[T]  \otimes \Omega[S] \cong \bigcup_{i=1}^N \Omega[R_i].$$
\end{prop}

In this context we also call an edge of a tree $T$ a colour of $T$. Let $P$ be a face of a shuffle $R$ of $S$ and $T$. We say that a colour $t$ of a tree $T$ \emph{appears} in $P$ if there is at least one edge $t \otimes s$ of $P$ for some colour $s$ of $S$. 

\begin{definition}
If there is a vertex $v=\{s_1,...,s_m\}$ of $S$ with output $s$, a vertex $w=\{t_1,...,t_n\}$ of $T$ with output $t$ and a shuffle $R$ such that $v\otimes t=\{(s_1,t),...,(s_m,t)\}$ and $s_i \otimes w=\{(s_i, t_1),....,(s_i,t_n)\}$ are vertices of $R$, then we form a new shuffle $R'$ which is a tree with
\begin{itemize}
	\item the underlying set  $R'= \{(s, t_1),..., (s, t_n)\} \cup R\setminus \{(s_1,t),...,(s_m,t)\}$,
	\item the unique partial order determined by $(s,t)\leq (s', t')$ in $R'$ if and only if $s\leq s'$ in $S$ and $t\leq t'$ in $T$,
	\item the set of leaves $L(R')$ of $R'$ being the same as the set leaves $L(R)$ of $R$.
\end{itemize}

We say that $R'$ is obtained from $R$ by a \emph{percolation step}.
\end{definition}

\begin{example} \label{percolation}
	Here is an example of a percolation step written in the form $R \to R'$ for the case where $S$ and $T$ are corollas with two and three inputs respectively: 
	$$ 
	
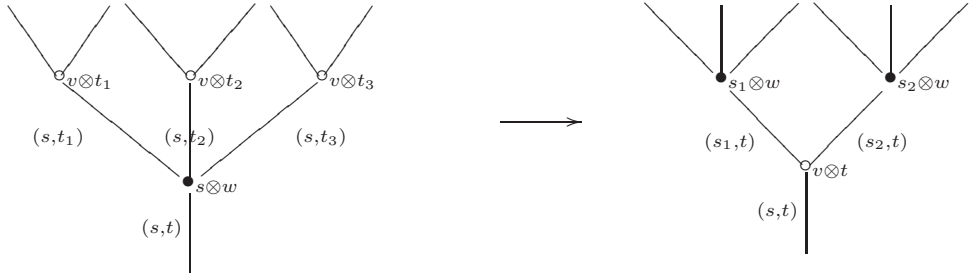
\xymatrix@R=10pt@C=6pt{
		& & &&&&&&& \\
		& & &&&&&&& \\
		& *=<3.1pt>{\quad \,\,\,\,\, \circ_{v\otimes t_1}} \ar@{-}[uul]  \ar@{-}[uur] && *=<3.1pt>{\quad \,\,\,\,\,  \circ_{v\otimes t_2}} \ar@{-}[uur]  \ar@{-}[uul] && *=<3.1pt>{\quad \,\,\,\,\,  \circ_{v\otimes t_3}} \ar@{-}[uul]  \ar@{-}[uur]  && \\
		& _{(s, t_1)} & &_{(s, t_2)}&&_{(s, t_3)}&& && \\
		&&&*{\quad \,\,\, \bullet_{s\otimes w}}  \ar@{-}[uull] \ar@{-}[uu] \ar@{-}[uurr]   &&&&\\
		&&&&&&&&& \\
		&&&*=0{}\ar@{-}[uu]^{(s, t)} &&&&
	}
	\xymatrix@R=10pt@C=6pt{ \\ \\ \\ \ar[rrr] & && \\ }
	\xymatrix@R=10pt@C=6pt{
		&&&&&&& \\
		&&&&&&& \\
		&& *{\quad \quad  \bullet_{s_1\otimes w}} \ar@{-}[uul]  \ar@{-}[uu] \ar@{-}[uur] && *{\quad \quad  \bullet_{s_2\otimes w}} \ar@{-}[uu]  \ar@{-}[uur] \ar@{-}[uul]  & \\
		&&&&& \\
		&&&*=<3.1pt>{\quad \, \, \circ_{v\otimes t}}  \ar@{-}[uul]^{(s_1, t)}  \ar@{-}[uur]_{(s_2, t)}   &&\\
		&&&&& \\
		&&&*=0{}\ar@{-}[uu]^{(s, t)} &&
	}
	$$
	A particular case is a percolation of a stump, the only case where a vertex of type $s\otimes w$ vanishes. Here is an example when $S$ is a corolla with no inputs: 
	$$  
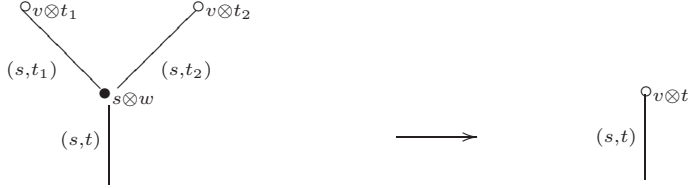
\xymatrix@R=10pt@C=12pt{
		&& *=0{\quad \quad  \circ_{v\otimes t_1}}  && *=0{\quad \quad  \circ_{v\otimes t_2}}  & \\
		&&&&& \\
		&&&*{\quad \,\, \bullet_{s\otimes w}}  \ar@{-}[uul]^{(s, t_1)}  \ar@{-}[uur]_{(s, t_2)}   &&\\
		&&&&& \\
		&&&*=0{}\ar@{-}[uu]^{(s, t)} &&
	}
	\xymatrix@R=10pt@C=6pt{ \\ \\ \\  &&&& \ar[rrr]  &&& \\ }
	\xymatrix@R=10pt@C=12pt{
		& & &&&&&&& \\
		& & &&&&&&& \\
		&&&*=0{\quad \,\,\,  \circ_{v\otimes t}}    &&&&\\
		&&&&&&&&& \\
		&&&*=0{}\ar@{-}[uu]^{(s, t)} &&&&
	}
	$$
	If $s$ is the output of a stump $v$ of $S$ and $t$ is the output of a stump $w$ of $T$ (i.e. they are minimal elements which are not leaves), then $(s,t)$ is a also the output of the black stump $s\otimes w$ which can turn into the white stump  $v\otimes t$ by a percolation step. Example in which $S$ and $T$ are both corollas with no inputs:
	$$
	
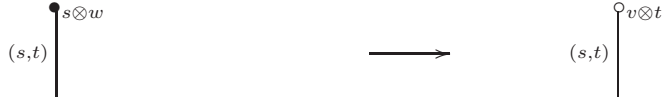
\xymatrix@R=10pt@C=12pt{
		&&&*=0{\quad \,\,\, \bullet_{s\otimes w}}    &&&&\\
		&&&&&&&&& \\
		&&&*=0{}\ar@{-}[uu]^{(s, t)} &&&&
	}
	\xymatrix@R=10pt@C=6pt{ \\ \ar[rrr] & && \\ }
	\xymatrix@R=10pt@C=12pt{
		&&&*=0{\quad \,\,\, \circ_{v\otimes t}}    &&&&\\
		&&&&&&&&& \\
		&&&*=0{}\ar@{-}[uu]^{(s, t)} &&&&
	}
	$$
\end{example}

Let $S$ and $T$ be trees, let $r_S$ and $r_T$ be the roots of $S$ and $T$ respectively, and $L(T)=\{l_1,...,l_m\}$ be the set of leaves of $T$. We let $S\otimes t$ (resp. $s\otimes T$) be a tree isomorphic to $S$ (resp. $T$) with the underlying set $S\times \{t\}$ (resp. $\{s\}\times T$). We may construct all shuffles of $S$ and $T$ inductively using percolation steps. 
We define $$R_1 = (r_S \otimes T) \circ (S \otimes l_1, S \otimes l_2, \ldots S \otimes l_m)$$ to be the shuffle  obtained by grafting copies of $S$ on top of $T$.

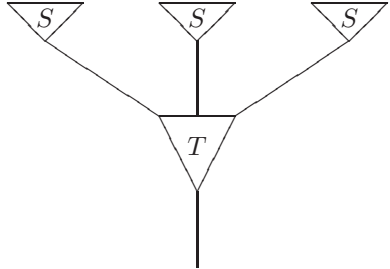
\[
\xy<0.1cm, 0cm>:
(0,0)*=0{}="1";
(0,10)*=0{}="2";
(-0,20)*=0{}="a";
(-5,20)*=0{}="3";
(5,20)*=0{}="4";
(-20,30)*=0{}="5";
(20,30)*=0{}="6";
(0,30)*=0{}="b";
(0,16)*=0{T};
"1";"2" **\dir{-};
"2";"3" **\dir{-};
"2";"4" **\dir{-};
"3";"4" **\dir{-};
"3";"5" **\dir{-};
"4";"6" **\dir{-};
"a";"b" **\dir{-};
(-25,35)*=0{}="9";
(-15,35)*=0{}="10";
(-5,35)*=0{}="11";
(5,35)*=0{}="12";
(15,35)*=0{}="15";
(25,35)*=0{}="16";
"5";"9" **\dir{-};
"5";"10" **\dir{-};
"9";"10" **\dir{-};
"b";"11" **\dir{-};
"b";"12" **\dir{-};
"11";"12" **\dir{-};
"6";"15" **\dir{-};
"6";"16" **\dir{-};
"15";"16" **\dir{-};
(-20,33)*=0{S};
(0,33)*=0{S};
(20,33)*=0{S};
\endxy\]

Note that the number of vertices of $R_1$ is finite, so there are finitely many shuffles of $S$ and $T$ and we obtain all shuffles from $R_1$ by letting the white vertices percolate towards the root in all possible ways.

\begin{definition} \label{shuffles}
	If a shuffle $R'$ is obtained from $R$ by a percolation step we write $R\preceq R'$ and say that $R$ is an \emph{immediate predecessor} of $R'$.  This defines a natural partial order on the set of all shuffles of $S$ and $T$ with $R_1$ being the unique minimal element. We call this the \emph{right percolation poset} of $S$ and $T$. Note that there is a unique maximal element in this partial set, namely the shuffle $R_N$ obtained by grafting copies of $T$ on top of $S$. The reverse partial set is called the \emph{left percolation poset}. 
\end{definition}

\subsection{Extension sets}
For the whole section, let us consider two trees $S$ and $T$ such that they are both open or one of them is linear. We assume that $S$ has a root vertex $v$ with inputs $l_1$, $l_2$, \ldots, $l_m$ such that $l_2$, $l_3$, \ldots, $l_m$ are leaves. Hence $S$ has a root face $\partial_v S$. We denote by $w$ the bottom vertex of $T$, and by $r_S$ and $r_T$ the root of $S$ and $T$, respectively. We fix an arbitrary total order $R_1 \preceq R_2 \preceq ... \preceq R_N$ extending the right percolation partial order on the set of shuffles.

\begin{prop} \label{prop extension 1}
	Let $R_i$ be  a shuffle of $S\otimes T$ with the bottom vertex $r_S\otimes w$. Let $A_0$ be the dendroidal subset of $\Omega[R_i]$ such that the missing faces with respect to $A_0$ are those for which
	\begin{itemize}
		\item all colours of $T$ appear, 
		\item all colours of $\partial_{v}S$ appear and 
		\item there is at least one edge which is not an edge of $R_j$, for each $j<i$. 
\end{itemize}
The inclusion $A_0 \to \Omega[R_i]$ is an inner dendroidal anodyne extension. 
\end{prop}

\begin{proof}
Let us define
\[
	X=\{x\in T\mid v\otimes x \textrm{ appears in } R_i  \}.
\]
We will show that for every missing face $P$, the set
\[
X_P=\{x\in X \mid (l_j, x) \textrm{ appears in } P \textrm{ for some } j=1,...,m\}
\]
is non-empty. To show this we consider occurrences of the colour $l_1$ in the shuffle $R_i$. We consider two cases.
 
\emph{Case 1.} 
Let us assume there is an edge $(l_1, t)$ in $R_i$, with $t\in T$, which is an input of a black vertex (i.e. a vertex of the form $l_1 \otimes u$ for some vertex $u$ of $T$). Along the branch from that edge to the root of the shuffle there must be an edge $(l_1, x)$ which is the output of a black vertex and an input of a white vertex $v\otimes x$ for some $x\in T$.  
By definition of the set $X$, we have $x\in X$. 

The following picture of the relevant part of the tree illustrates the situation. 
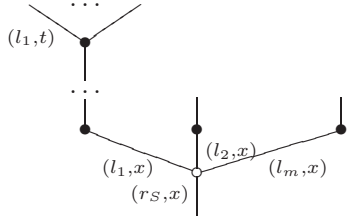
\[  \xymatrix@R=10pt@C=12pt{
	& & \ldots &&&&&& \\
	& & *=0{\bullet} \ar@{-}[ul]^{(l_1,t)}  \ar@{-}[ur] &&&&&& \\
	& & \ldots \ar@{-}[u] &&&&&& \\
		& & *=0{\bullet} \ar@{-}[u] &&  *=0{\bullet} \ar@{-}[u]  &&  & *=0{\bullet} \ar@{-}[u] & \\
	&&&& *=<3.1pt>{\circ}  \ar@{-}[ull]^{(l_1,x)} \ar@{-}[u]_{(l_2,x)}  \ar@{-}[urrr]_{\quad (l_m, x)}    &&&&\\
	&&&&*=0{}\ar@{-}[u]^{(r_S, x)} &&&&
}
\]
Since $l_2,...,l_m$ are leaves of $S$, all edges $(l_j,x)$ are inputs of a white vertex $v \otimes x$ and outputs of black vertices. Hence the shuffle $R_i$ has a predecessor $R_k$, $k<i$, which does not contain $(l_j, x)$, $j=1,...,m$ (to obtain $R_k$ we can just apply an inverse percolation to $R_i$ at this white vertex $v \otimes x$). By the description of the missing faces at least one of these edges must appear in $P$. 

\emph{Case 2.} Let us assume that for every $t\in T$ edge $(l_1, t)$ is a leaf above a white vertex or connects two white vertices in the shuffle $R_i$. In this case the colour $l_1$ appears only on these edges. Colour $l_1$ must appear in $P$, so $P$ must contain such an edge $(l_1, t)$. This shows that $X_P$ is non-empty.\\

We now return to the proof of the proposition. Note that for a missing face $P$ with an inner edge $(r_S,x)$, $x\in X_P$, the face $\partial_{(r_S,x)} P$ is also missing as the colour $x$ appears on the edge $(l_j, x)$ by definition of $X_P$. We define 
\[
F=\{\partial_{f} P\to P \mid  P\not\in A_0, f=(r_S,x) \text{ inner edge of } P, x\in X_P \}
\]
and claim that $F$ is an extension set. Axioms (F1) and (F2) are satisfied because $F$ contains only inner elementary face maps. Axiom (F3) obviously holds because belonging of $\partial_f$ to $F$ depends only on the edge $f$. Analogously, every extension $\partial_f \in F$ appears only in diagrams of the form 
\[ \xymatrix{\partial_f P \ar[d]_{\partial_f} \ar[rr]^{\partial_g} &&  P' \ar[d]^{\partial_f} \\  P \ar[rr]_{\partial_g} && P\cup P',} \] so Axiom (F4) holds, too.  

Finally, to check Axiom (F5), note that for a missing face $P$ and $x$ in $X_P$, the edge $(r_S, x)$ is inner in $R_i$ and lies between $(l_j, x)$ and $(r_S, r_T)$. Since $P$ is missing, the edge $(r_S, r_T)$ must appear (so that $r_T$ appears in $P$) and the edge $(l_j, x)$ must appear because $x$ is in $X_P$. 

The situation can be again pictured with the relevant part of the tree. 
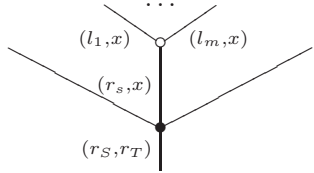
\[ 
 \xymatrix@R=10pt@C=12pt{
	& &  & &\ldots   &&&& \\
	& &  && *=<3.1pt>{\circ} \ar@{-}[ul]^{(l_1, x)}  \ar@{-}[ur]_{(l_m, x)} &&&& \\
	& &  & & &&&& \\
	&&&& *=0{\bullet}   \ar@{-}[uulll] \ar@{-}[uu]^{\quad (r_s, x)}  \ar@{-}[uurrr]    &&&&\\
	&&&&*=0{}\ar@{-}[u]^{(r_S, r_T)} &&&&
}
\]
 If $(r_S,x)$ appears in $P$, it must be an inner edge and $\F_F(P)$ is not empty. If $(r_S, x)$ does not appear in $P$, we can extend $P$ with the edge $(r_S,x)$ to obtain $P'$ such that $\partial_{(r_S,x)} P' = P$, so $\E_F(P)$ is not empty.  By Theorem \ref{theorem method}, it follows that $A_0\to \Omega[R_i]$ is an inner dendroidal anodyne extension. 
\end{proof}

\begin{definition}
	Let $R_i$ be a shuffle of $S\otimes T$ with the bottom vertex $v\otimes r_T$. We say that a face $R$ of $R_i$ is \emph{essential} if it contains all the edges of $R_i$ of the form $l_j \otimes t$ for $j\in \{2,...,m\}$ and $t\in T$. 
\end{definition}

\begin{definition} \label{def T-top}
	Let $R_i$ be a shuffle of $S\otimes T$ with the bottom vertex $v\otimes r_T$ and $R$ an essential face of $R_i$. The \emph{$T$-covering set of $R$} is a subset $Y$ of $T$ such that $x$ is in $Y$ if there is a leaf $s\otimes x$ of $R$ for $l_1\leqslant s$. A subset $X$ of $Y$ consisting of maximal elements with respect to the order in $T$ is called the \emph{$T$-top of $R$}. 
\end{definition}

\begin{example} \label{example shuffle2} Here is another example of a shuffle of the same two trees as in Example \ref{example shuffle}
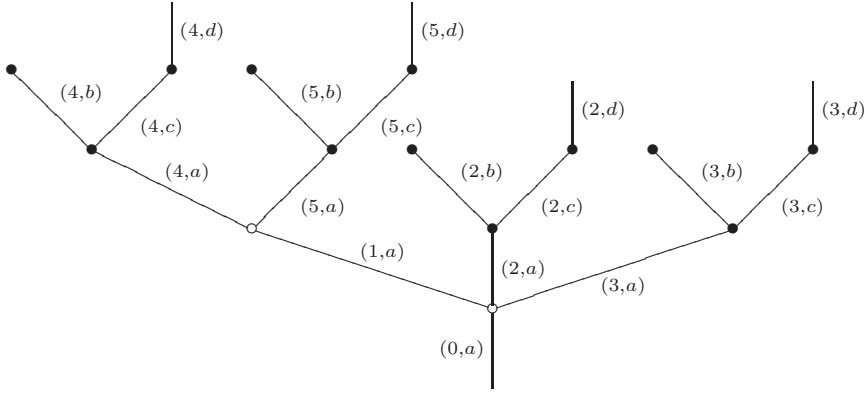
	\[
	\xymatrix{
		&&&&&&&&&& \\
		*=0{\bullet} && *=0{\bullet} \ar@{-}[u]_{(4,d)}   & *=0{\bullet} && *=0{\bullet} \ar@{-}[u]_{(5,d)} &&&&&& \\
		&*=0{\bullet} \ar@{-}[ul]_{(4,b)}  \ar@{-}[ur]_{(4,c)} & &  & *=0{\bullet} \ar@{-}[ul]_{(5,b)} \ar@{-}[ur]_{(5,c)}  
		& *=0{\bullet}  && *=0{\bullet}  \ar@{-}[u]_{(2,d)}   &*=0{\bullet}  && *=0{\bullet}  \ar@{-}[u]_{(3,d)}   & \\
		&&& *=<3.1pt>{\circ} \ar@{-}[ull]_{(4,a)} \ar@{-}[ur]_{(5,a)} &&& *=0{\bullet} \ar@{-}[ul]_{(2,b)} \ar@{-}[ur]_{(2,c)} &&&  *=0{\bullet} \ar@{-}[ul]_{(3,b)} \ar@{-}[ur]_{(3,c)} &&  \\
		&&&&&&*=<3.1pt>{\circ} \ar@{-}[ulll]_{(1,a)} \ar@{-}[u]_{(2,a)} \ar@{-}[urrr]_{(3,a)}  & &&\\
		&&&&&&*=0{}\ar@{-}[u]^{(0,a)} &&
	} \]
and an example of its essential face 
	
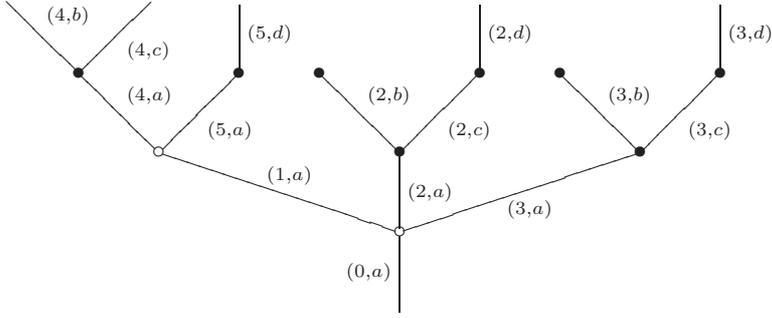
\[
	\xymatrix{
		 &&&&&&&&&& \\
		&*=0{\bullet} \ar@{-}[ul]_{(4,b)} \ar@{-}[ur]_{(4,c)}  &  & *=0{\bullet} \ar@{-}[u]_{(5,d)}   
		& *=0{\bullet}  && *=0{\bullet}  \ar@{-}[u]_{(2,d)}   &*=0{\bullet}  && *=0{\bullet}  \ar@{-}[u]_{(3,d)}   & \\
		&& *=<3.1pt>{\circ} \ar@{-}[ul]_{(4,a)} \ar@{-}[ur]_{(5,a)} &&& *=0{\bullet} \ar@{-}[ul]_{(2,b)} \ar@{-}[ur]_{(2,c)} &&&  *=0{\bullet} \ar@{-}[ul]_{(3,b)} \ar@{-}[ur]_{(3,c)} &&  \\
		&&&&&*=<3.1pt>{\circ} \ar@{-}[ulll]_{(1,a)} \ar@{-}[u]_{(2,a)} \ar@{-}[urrr]_{(3,a)}  & &&\\
		&&&&&*=0{}\ar@{-}[u]^{(0,a)} &&
	} \]
with the $T$-covering $Y=\{b,c,d\}$ and the $T$-top $X=\{b,d\}$. 
\end{example}

\begin{lemma} \label{lemma operation Y}
	Let $R_i$ be a shuffle of $S\otimes T$ with the bottom vertex $v\otimes r_T$ and $R$ an essential face of $R_i$. The $T$-covering set $Y$ of $R$ has the property that every branch from a leaf to the root of $R_i$ has at least one edge of the form $(s, x)$, $x\in Y$. In particular, for the $T$-top $X$, we have an operation $(X; r_T)$ in $T$. 
\end{lemma}

\begin{proof}
	For $R_i$ the statement is true since $Y$ is the set $L(T)$ of leaves of $T$. Each essential tree $R$ is obtained from $R_i$ by a sequence of inner and top face maps above the edge $(l_1, r_T)$. By the dendroidal relations, we know that to obtain $R$, we may first perform top faces and then inner faces. Hence it is enough to prove that the stated property of $R$ does not change as we contract an inner edge or chop off a top vertex above $(l_1, r_T)$.  For inner face maps, the statement is obvious as the set of leaves of $R$ does not changes, so the $T$-covering set $Y$ does not change. By chopping off a black top vertex,  the $T$-covering set $Y$ does not change. When chopping off a white top vertex, it might happen that the set $Y$ changes, but replace the inputs of one vertex in $T$ with the output of that vertex, so the stated property still holds. 
\end{proof}

\begin{prop} \label{shuffle prop}
	Let $R_i$ be a shuffle of $S\otimes T$ with the bottom vertex $v\otimes r_T$, $R$ an essential face of $R_i$ and $X$ the $T$-top of $R$. Let $A$ be the dendroidal subset of $\Omega[R]$ such that the missing faces are all those faces for which: 
	\begin{itemize}
		\item all edges of $R$ of the form $(s, t)$ appear, for $t\in T$, $s\in S$ and $l_1\leqslant s$, 
		\item all colours of $T$ appear, 
		\item all colours of $S$ appear, 
		\item there is at least one edge which is not an edge of $R_k$, for each $k<i$.
	\end{itemize}
 In addition, if $R$ contains the edge $(l_1, r_T)$, assume that the unique maximal face $R''$ of $R$ having $(l_j, x)$ as leaves, for all $x\in X$ and $j\in \{2,\ldots, m\}$, is not missing. Then, the inclusion $A \to \Omega[R]$ is a covariant dendroidal anodyne extension. 
\end{prop}

\begin{example} We illustrate the tree $R''$ for the essential face described in the Example \ref{example shuffle2}
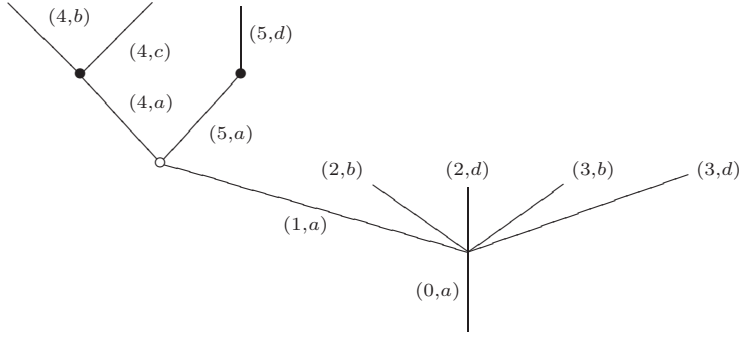
\[ 	
	\xymatrix{
		 &&&&&&&&& \\
		&*=0{\bullet} \ar@{-}[ul]_{(4,b)} \ar@{-}[ur]_{(4,c)}  &  & *=0{\bullet} \ar@{-}[u]_{(5,d)}   &&&&& \\
		&& *=<3.1pt>{\circ} \ar@{-}[ul]_{(4,a)} \ar@{-}[ur]_{(5,a)} &&_{(2,b)}  & _{(2,d)}  & _{(3,b)}& _{(3,d)}&&\\
		&&&&&*=0{} \ar@{-}[ulll]^{(1,a)}    \ar@{-}[ul]\ar@{-}[u]   \ar@{-}[ur] \ar@{-}[urr]  & &\\
		&&&&&*=0{}\ar@{-}[u]^{(0,a)} &
	}
\] 
with the $T$-top $X=\{b,d\}$.
\end{example}

\begin{proof} Recall the definition of an $X$-vertex from Definition \ref{X-vertex}.
We say that an elementary face map $\partial_f P\to P$ is an \emph{$X$-face map} if $f$ is
\begin{itemize}
	\item an inner edge $(l_j, x)$, $x\in X, j\in \{2,...,m\}$ or
	\item a top vertex  $l_j \otimes w$ such that $w$ is an $X$--vertex and $j\in \{2,...,m\}$.
\end{itemize}
Let $F$ be a set consisting of $X$-face maps $\partial_f P\to P$ such that $P$ and $\partial_f P$ are missing faces of $R$ with respect to $A$.  We claim that $F$ is an extension set. We now check the axioms. 

\begin{enumerate}
\item[(F1)] 
The Forbidden Pair Axiom follows immediately from the definition of the set $F$. Indeed, since $X$ is an operation there are no two top face maps with the same output. Thus, there is no bad pair of extensions in $F$. Similarly, there is no pair of adjacent or mixed face maps in $F$. 

\item[(F2)] 
To show The Bad Pair Axiom, let $\partial_g$ be an extension of a missing face $P$. A pair $(\partial_f , \partial_g)$ of extensions of $P$ with $\partial_f \in F$ is bad only if
\begin{itemize}
	\item $f$ is a top vertex of the form $l_j\otimes w$, where $w$ is an $X$-vertex with the output $x$ and $j\in\{2,\ldots,m\}$,
	\item $g$ is a top vertex with the same output $(l_j,x)$. 
\end{itemize}	
Since $X$ is an operation, there is at most one such $f$ as $w$ is uniquely determined as the set of all elements of $X$ above $x$. 

\item[(F3)] 
Next, to check The Face Closure Axiom, let us consider the following commutative diagram of elementary face maps: 
\[ \xymatrix{\partial_g\partial_f P \ar[d]_{\partial_f} \ar[rr]^{\partial_g} && \partial_f P \ar[d]^{\partial_f} \\  \partial_gP \ar[rr]_{\partial_g} && P.} \]
If we assume that $P$, $\partial_f P$ and $\partial_g\partial_P$ are missing faces, then $\partial_g P$ is also missing because it contains all edges of $\partial_g \partial_f P$ and missing faces are determined by their set of edges. 
We also need to prove that if $P$, $\partial_f P$ and $\partial_g P$ are missing faces, then $\partial_g \partial_f P$ is missing, too. The edges deleted from $P$ to obtain $\partial_f P$ and from $\partial_f P$ to obtain $\partial_g \partial_f P$ are of the form $(l_j, x)$, $x\in X$, $j\in \{2,...,m\}$, so:
\begin{itemize}
	\item $\partial_g\partial_f P$ contains all edges of the form $(s,t)$ of $R$ for $l_1 \leqslant s$, since the same is true for $P$; 
	\item the only colour of $T$ that might have been erased is some $x$ in $X$, but $P$ and hence $\partial_g\partial_f P$ contains $(s, x)$ (with $l_1\leqslant s$) by the definition of the set $X$; 
	\item the only colour of $S$ that might have been erased is $l_j$ for $j\in \{2,\ldots, m\}$, but these colours must appear in $\partial_g \partial_f P$ because we have not used root faces by which we would erase all parts of $l_2\otimes T$, \ldots, $l_m\otimes T$;
	\item there is at least on edge in $\partial_g \partial_f P$ which is not edge of $R_k$ for $k<i$ because the same is true for $P$. 
\end{itemize}
Belonging of $\partial_f$ to $F$ depends only on the set $f$, so any side of the above square belongs to $F$ if and only if the opposite side belongs to $F$.

\item[(F4)] 
For The Extension Closure Axiom, we first note that any extension of a missing face is missing. Let $\partial_g \colon P \to P'' $ be an elementary face map and $\partial_f \colon P \to P'$, $\partial_f \in \E_F(P)$. If $\partial_f$ and $\partial_g$ are not elementary face maps corresponding to top vertices with the same output, then $(\partial_f, \partial_g)$ is a good pair, and the statement follows again because any side of the square 
\[ \xymatrix{P \ar[d]_{\partial_f} \ar[rr]^{\partial_g} && P' \ar[d]^{\partial_f} \\  P'' \ar[rr]_{\partial_g} && P'\cup P''.} \]
belongs to $F$ if and only if the opposite side belongs to $F$, too. 

The pair $(\partial_f, \partial_g)$ is bad if and only if $f$ and $g$ are top vertices with the same output. In Remark \ref{extension sequence remark}, we have described that the extension sequence $\partial_{f_1},..., \partial_{f_r}$ satisfies \[f_1\cup \ldots \cup f_n = f.\] Since belonging to $F$ depends only on the set of edges being erased, all these maps are in $F$. 

\item[(F5)] 
Finally, to show The Existence Axiom, let $P$ be a missing face such that $\F_F(P)$ is empty.
If $R$ does not contain the edge $l_1\otimes r_T$, then $l_1$ must appear on some other edge of $R$. Hence it must also appear on some other edge of $R_i$, so $(l_1, r_T)$ is an output of a black vertex in $R_i$. From this we see that there is a percolation step in which edges $(l_1, r_T)$, $(l_2,r_T)$, \ldots, $(l_m, r_T)$ appear in $R_i$, so the missing face $P$ must have at least one edge of the form $(l_j, r_T)$ for $j\geq 2$. Let us fix one such $j$. By the assumption that $\F_F(P)$ is empty, $P$ has no inner edges of the form $(l_j, x)$ with $x\in X$ and it has no top vertices of the form $l_j\otimes w$ with $w$ being an $X$-vertex. Hence there is a leaf $(l_j, y)$ of $P$ such that $y\not \in X$. Since $(X; r_T)$ is an operation of $T$, the set $X$ has the property that for edge $y$ of $T$ there
either exists $x\in X$ such that $y\leq x$ or $x \leq y$ or there exists a stump $w$ of $T$ with an output $x$ such that $y\leq x$.

In the first case, since $P$ has no inner edges $(l_j,x)$, $x\in X$, there must exist $x\in X$ such that $y \leq x$. There exists a unique face $P'$ with a top vertex $l_j\otimes w$ such that $w$ is an $X$-vertex, $(l_j,y)$ is the output and $(l_j,x)$ is one of the leaves of $l_j\otimes w$ and such that $\partial_{l_j\otimes w}P' =P$. 

Similarly, in the second case, there exists a unique face $P'$ with a top vertex $l_j\otimes w$ with $w$ the stump with the output $(l_j,y)$ and such that $\partial_{l_j\otimes w}P' =P$. In any case, we conclude $\E_F(P)$ is not empty. 

If $R$ contains the edge $(l_1, r_T)$, the conclusion that each missing face has at least on edge of the form $(l_j, r_T)$ for $j\geq 2$ follows from the assumption that $R''$ is not missing. 
\end{enumerate}	

This ends the proof as the result follows by Theorem  \ref{theorem method}.\end{proof}

\subsection{The pushout-product property for dendroidal sets.}

\begin{theorem} \label{PushoutProductStab}
Let $S$ and $T$ be trees, let $v$ be the bottom vertex of $S$ with inputs $l_1,l_2,...,l_m$ such that $l_2,...,l_m$ are leaves. If $S$ or $T$ is linear or both $S$ and $T$ are open trees, then the morphism
$$\Lambda^{v}[S] \otimes \Omega[T] \cup \Omega[S] \otimes \partial\Omega[T] \to \Omega[S]\otimes \Omega[T]$$ is a stable anodyne extension.
\end{theorem}

\begin{remark}
	The conditions on $S$ and $T$ ensure that the stated morphism is a normal monomorphism and we make the same assumptions following Erratum, \cite{dSet model hom op}. 
\end{remark}

\begin{remark}
The following proof applies equally if $S$ is a corolla or a tree with more than one vertex. 
If we consider the case where $S$ is linear, then $m=1$. 
\end{remark}

\begin{proof}

If $T=\eta$ the statement is equivalent to saying that the horn inclusion $\Lambda^{v} [S] \to \Omega[S]$ is a stable anodyne extension, which is true by definition.  Hence we assume that $T$ has at least one vertex. We denote by $r_S$ (respectively $r_T$) the root of $S$ (respectively $T$).

We fix a total order $R_1 \preceq R_2 \preceq ... \preceq R_N$ extending the right percolation partial order (see Definition \ref{shuffles}). 
Let $B_0=\Lambda^{v}[S] \otimes \Omega[T] \cup \Omega[S] \otimes \partial\Omega[T]$ and we define $B_i = B_{i-1} \cup \Omega[R_i]$. The assumptions on $T$ and $S$ imply that all maps $B_{i-1}\to B_i$ are monomorphisms. 

If $T$ has no leaves, then $R_1$ is $r_S \otimes T$ and $B_1=B_0$. In that case we will show that the inclusions $B_{i-1} \to B_{i}$ are stable anodyne extensions for all $i=2,...,N$. In the case $T$ has at least one leaf, we will show that the inclusions $B_{i-1} \to B_{i}$ are stable anodyne extensions for all $i=1,2,...,N$.

If we denote $A:=B_{i-1}\cap \Omega[R_i]$ then we have a pushout diagram
\begin{equation*}
\xymatrix {A \ar[d] \ar[r] & B_{i-1} \ar[d] \\ \Omega[R_i] \ar[r] & B_i.}
\end{equation*}
From this it follows that it is enough to show that $A \to \Omega[R_i]$ is a stable anodyne extension, for all $i$. 

We distinguish two cases. If the bottom vertex of $R_i$ is black, then it is clear that the assumptions of Proposition \ref{prop extension 1} are satisfied, so $A\to \Omega[R_i]$ is an inner anodyne extension.

To deal with the case when the bottom vertex of $R_i$ is white ($v\otimes r_T$), we introduce some notation. Let $R'_i$ be the maximal face of $R_i$ with the edge $(l_1, r_T)$ being the root. Another way of looking at this is that $R_i$ is obtained by grafting $R'_i, l_2\otimes T,..., l_m\otimes T$ on the corolla with the vertex $v\otimes r_T$, i.e.
\[ R_i=(v\otimes r_T) \circ (R'_i, l_2\otimes T,..., l_m\otimes T). \]
Let us consider the family 
\[
\mathcal{H} =\{ R' \mid R'\in\mathrm{Sub}(R'_i), \, \text{ root of } R' \text{ is } (l_1, r_T) \}
\]
 of all faces of $R_i'$ with the root $(l_1, r_T)$. For each such face $R'$, we can form a face $R=f(R')$ of $R_i$ by grafting $R', l_2\otimes T,..., l_m\otimes T$ on the corolla with the vertex $v\otimes r_T$, i.e.
\[ f(R')=(v\otimes r_T)\circ (R', l_2\otimes T,..., l_m\otimes T). \]
We also consider the family of all such trees: 
\[
\mathcal{G} = \{ f(R') \mid R'\in\mathrm{Sub}(R'_i), \, \text{ root of } R' \text{ is } (l_1, r_T) \}. 
\]

The idea is to proceed in two steps. In the first step we add to the filtration all missing faces $\partial_{(l_1, r_T)} R$, for $R\in \mathcal{G}$, and in the second step we add all missing faces  $R'\in \mathcal{H}$ and $R\in \mathcal{G}$.

\subsection*{Step 1.} Let us denote by $B'_{i-1}$ the union of $B_{i-1}$ with the representables of all $\partial_{(l_1, r_T)} R$, for $R\in \mathcal{G}$. 
We will show that the inclusion $B_{i-1} \to B'_{i-1}$ is a covariant anodyne extension. Let $K$ be the number of vertices of $R'_i$ and let us define inductively a filtration
$$B_{i-1}=C_0 \subseteq C_1 \subseteq .... \subseteq C_k \subseteq ... \subseteq C_K=B'_{i-1}$$
where $C_k$ is the union of $B_{i-1}$ with the representables of all $\partial_{(l_1, r_T)} R$, for all $R'\in \mathcal{H}$ with at most $k$ vertices.
For any tree $R'\in \mathcal{H}$ and $R=f(R')$, the inclusion 
\[C_{k-1} \cap \Omega[\partial_{(l_1, r_T)} R] \to \Omega[\partial_{(l_1, r_T)} R] \] 
satisfies the assumptions of Proposition \ref{shuffle prop}, so it is a covariant anodyne extension.  
Since we have a pushout diagram
\[\xymatrix {\coprod ( C_{k-1} \cap \Omega[\partial_{(l_1, r_T)} R]) \ar[d] \ar[r] & C_{k-1} \ar[d] \\ \coprod \Omega[\partial_{(l_1, r_T)} R] \ar[r] & C_k} \]
where the coproduct is taken over all faces $R\in \mathcal{G}$, we conclude that the inclusion $C_{k-1}\to C_k$ is a covariant anodyne extension for each $k$, so $B_{i-1} \to B_{i-1}'$ is one, too.


\subsection*{Step 2.} We next show that the inclusion $B'_{i-1}\to B_i$ is stable anodyne.
Let us define inductively a filtration \[B'_{i-1}=D_0 \subseteq D_1 \subseteq .... \subseteq D_k \subseteq ... \subseteq D_K=B_i\]
where $D_k$ is the union of $B'_{i-1}$ with the representables of all $R$, for all possible faces $R'$ with at most $k$ vertices.

For a missing face $R'\in \mathcal{H}$ (i.e. if $\Omega[R']\to \Omega[R_i]$ does not factor through $B'_{i-1}$), we consider the tree $R''$ obtained by grafting $R'$ on the leaf $l_1\otimes r_T$ of the corolla with the root $r_S\otimes r_T$ and the leaves (other than $(l_1, r_T)$) of the form $(l_j, x), x\in X, j\in \{2,...,m\}$, where $X$ is the $T$-top of $f(R')$ (see Definition \ref{def T-top}). Let us call $u$ the unique vertex of $R''$ attached to the root. The inclusion $\Lambda^u[R''] \to \Omega[R'']$ is a stable anodyne extension and $\Lambda^u[R'']$ factors through $D_{k-1}$ where $k$ is the number of vertices of $R'$.

If $\Omega[R']\to \Omega[R_i]$ does not factor through $B'_{i-1}$, then $\Omega[R'']\to \Omega[R_i]$ also does not factor through $B'_{i-1}$, so the inclusion 
\[A:=(D_{k-1} \cup \Omega[R''])\cap\Omega[ R]\] satisfies the assumptions of Proposition \ref{shuffle prop} and hence it is a covariant anodyne extension. 
Since we have a pushout diagram
\[\xymatrix {\coprod D_{k-1} \cap \Omega[ R] \ar[d] \ar[r] & D_{k-1} \ar[d] \\ \coprod \Omega[ R] \ar[r] & D_k} \]
where the coproduct is taken over all missing faces $R'\in\mathcal{H}$, we conclude that $D_{k-1}\to D_k$ is a stable anodyne extension for each $k$. Thus, $B'_{i-1}\to B_i$ is  a stable anodyne extension and the proof is complete.
\end{proof}

\begin{remark} \label{PushoutProductOper}
	Our method also applies to show Proposition 9.2. in \cite{inn Kan in dSet}. Let $T$ and $S$ be two trees and $e$ an inner edge of the tree $S$. If both $S$ and $T$ are open trees or one of them is linear, then the morphism
	$$\Lambda^e[S] \otimes \Omega[T] \cup \Omega[S] \otimes \partial\Omega[T] \to \Omega[S]\otimes \Omega[T]$$ is an inner anodyne extension.
	
	We use the filtration given by adding shuffles one by one following the left percolation poset. Let $v$ be the unique vertex of $S$ such that $e$ is the input of. $v$. For a fixed shuffle $R_i$ we define
	\[X=\{x \in T \mid v \otimes x \textrm{ is a vertex of } R_i\} \]
	and the extension set is then given by (inner elementary face maps)
	\[
		F=\{\partial_{(x, e)}\colon \partial_{(x, e)} P \to P \mid  P \textrm{ missing face }, x\in X\}.
	\]
\end{remark}

\end{document}